\documentclass[10pt,reqno]{amsart}
\usepackage{hyperref}
\usepackage{latexsym,amsmath}
\usepackage{enumerate}
\usepackage{amsfonts}
\usepackage{amssymb}
\usepackage{geometry}
\usepackage{latexsym}
\usepackage{fixmath}
\usepackage[prependcaption]{todonotes}
\usepackage{tikz}

\usepackage[T1]{fontenc}
\usepackage{fourier}
\usepackage{bbm}
\usepackage{color}
\usepackage{fullpage}

\numberwithin{equation}{section}

\def\nfrac#1#2{{\textstyle\frac{#1}{#2}}}

\renewcommand\ln{\log}
\newcommand\disteq{\sim}

\newcommand{\fR}{\mathfrak R}

\newcommand{\va}{\vec a}

\newcommand{\vd}{\vec d}
\newcommand{\vD}{\vec D}

\newcommand{\vc}{\vec c}

\newcommand{\vX}{\vec X}

\renewcommand{\epsilon}{\eps}

\newcommand{\GG}{\mathbb G}

\newcommand\ALPHA{\vec\alpha}

\newcommand\vY{\vec Y}

\newcommand\vr{\vec r}
\newcommand\vm{\vec m}

\newcommand\GAMMA{{\vec\gamma}}

\newcommand\RHO{{\vec\rho}}

\newcommand\nix{\,\cdot\,}

\newcommand\dd{{\mathrm d}}

\newcommand\G{\vec G}

\renewcommand{\vec}[1]{\boldsymbol{#1}}

\newcommand\SIGMA{\vec\sigma}

\newcommand\aco[1]{#1}
\newcommand\csg[1]{#1}
\newtheorem{definition}{Definition}[section]

\newtheorem{remark}[definition]{Remark}
\newtheorem{theorem}[definition]{Theorem}
\newtheorem{lemma}[definition]{Lemma}
\newtheorem{proposition}[definition]{Proposition}
\newtheorem{corollary}[definition]{Corollary}

\newcommand\fC{\mathfrak{C}}

\newcommand\fp{\mathfrak{a}}
\newcommand\fr{\mathfrak{r}}

\newcommand\cA{\mathcal{A}}

\newcommand\cC{\mathcal{C}}

\newcommand\cG{\mathcal{G}}
\newcommand\cE{\mathcal{E}}

\newcommand\cS{\mathcal{S}}

\newcommand\cM{\mathcal{M}}

\newcommand\cP{\mathcal{P}}
\newcommand\cX{\mathcal{X}}
\newcommand\cY{\mathcal{Y}}

\newcommand\cZ{\mathcal{Z}}
\def\cR{{\mathcal R}}
\def\cE{{\mathcal E}}

\newcommand\ve{\vec e}
\newcommand\vu{\vec u}
\newcommand\vv{\vec v}
\newcommand\vx{\vec x}

\newcommand\vM{\vec M}
\newcommand\vy{\vec y}

\newcommand\eul{\mathrm{e}}
\newcommand\eps{\varepsilon}

\newcommand\NN{\mathbb{N}}

\newcommand\Erw{\mathbb{E}}
\newcommand{\vecone}{\vec{1}}

\newcommand{\Po}{{\rm Po}}
\newcommand{\Bin}{{\rm Bin}}

\newcommand\dTV{d_{\mathrm{TV}}}

\newcommand{\bink}[2] {{\binom{#1}{#2}}}

\newcommand\bc[1]{\left({#1}\right)}
\newcommand\cbc[1]{\left\{{#1}\right\}}
\newcommand\bcfr[2]{\bc{\frac{#1}{#2}}}
\newcommand{\bck}[1]{\left\langle{#1}\right\rangle}
\newcommand\brk[1]{\left\lbrack{#1}\right\rbrack}
\newcommand\scal[2]{\bck{{#1},{#2}}}

\newcommand\abs[1]{\left|{#1}\right|}

\newcommand\RR{\mathbb{R}}

\newcommand{\whp}{w.h.p.}

\newcommand{\Erdos}{Erd\H{o}s}
\newcommand{\Renyi}{R\'enyi}

\newcommand{\Bollobas}{Bollob\'as}

\newcommand{\Luczak}{\L uczak}
\newcommand{\Mezard}{M\'ezard}
\newcommand{\Kucera}{Ku\v{c}era}

\newcommand\pr{\mathbb{P}} 
\renewcommand\Pr{\pr} 

\newcommand\Lem{Lemma}
\newcommand\Prop{Proposition}
\newcommand\Thm{Theorem}
\newcommand\Def{Definition}
\newcommand\Cor{Corollary}
\newcommand\Sec{Section}

\usepackage{etoolbox}
\makeatletter
\patchcmd{\@maketitle}
  {\ifx\@empty\@dedicatory}
  {\ifx\@empty\@date \else {\vskip3ex \centering\footnotesize\@date\par\vskip1ex}\fi
   \ifx\@empty\@dedicatory}
  {}{}
\patchcmd{\@adminfootnotes}
  {\ifx\@empty\@date\else \@footnotetext{\@setdate}\fi}
  {}{}{}
\makeatother

\begin{document}

\title{Lower bounds on the chromatic number of random graphs}

\author{Peter Ayre, Amin Coja-Oghlan, Catherine Greenhill}

\address{Peter Ayre, {\tt peter.ayre@unsw.edu.au}, School of Mathematics and Statistics, UNSW Sydney, NSW 2052, Australia.}

\address{Amin Coja-Oghlan, {\tt amin.coja-oghlan@tu-dortmund.de}, TU Dortmund, Faculty for Computer Science, 12 Otto Hahn St, Dortmund 44227, Germany.}

\address{Catherine Greenhill, {\tt c.greenhill@unsw.edu.au}, School of Mathematics and Statistics, UNSW Sydney, NSW 2052, Australia.}

\thanks{Research of the first author supported in part by DFG CO 646.
Research of the third author supported by the Australian Research Council Discovery Project DP190100977.}

%\date{DRAFT: 12 May 2019}

\begin{abstract}
We prove that a formula predicted on the basis of non-rigorous physics arguments
[Zdeborov\'a and Krzakala: Phys.\ Rev.\ E  (2007)] provides a lower bound on the chromatic number of sparse random graphs.
The proof is based on the interpolation method from mathematical physics.
In the case of random regular graphs the lower bound can be expressed algebraically, while in the case of the binomial random we obtain a variational formula.
As an application we calculate improved explicit lower bounds on the chromatic number of random graphs for small (average) degrees.
Additionally, we show how asymptotic formulas for large degrees that were previously obtained by lengthy and complicated combinatorial arguments can be re-derived easily from these new results.
\hfill{\em MSC: 05C80}
\end{abstract}

\maketitle

\section{Introduction}\label{Sec_intro}

\subsection{Motivation and background}\label{Sec_background}
A most fascinating feature of combinatorics is how easy-to-state problems sometimes lead to deep and difficult mathematical challenges.
The random graph colouring problem is a case in point.
First mentioned in the seminal paper of \Erdos\ and \Renyi\ that started the theory of random graphs~\cite{ER},
the problem of finding the chromatic number of the binomial random graph $\GG(n,d/n)$ with a fixed average degree $d$ remains open  to this day.
It is, in fact, the single open problem posed in that seminal paper that still awaits a complete solution.
Nor is the chromatic number of the random $d$-regular graph, a conceptually simpler object, known
for all values of~$d$.
Nevertheless, the quest for the chromatic  number has contributed tremendously to the development of new techniques, some of which now count among the standard tools of probabilistic combinatorics~\cite{ShamirSpencer}.

A series of important papers contributed ever tighter bounds on the chromatic number of random graphs.
Straightforward first moment calculations show that for any $q\geq3$ and for $\GG$ either the binomial%
\footnote{\aco{Sometimes $\GG(n,d/n)$ is referred to as the \Erdos-\Renyi\ model. This is a slight misnamer as \Erdos\ and \Renyi~\cite{ER} actually worked with a uniformly random graph with a given number of edges.}}
random graph $\GG(n,d/n)$ or the random regular graph $\GG(n,d)$,
\begin{align}\label{eqFirstMmt}
\chi(\GG)&>q\quad\mbox{ \whp}&&\mbox{ if }\qquad\log q+\frac d2\log(1-1/q)<0.
\end{align}
To be precise, \eqref{eqFirstMmt} is obtained by computing the expected number of $q$-colourings%
\footnote{\aco{Throughout the paper the term $q$-colouring or just colouring refers to proper vertex colourings. That is, colours are assigned to the vertices of a graph such that no two adjacent vertices receive the same colour.}}
, which tends to zero as $n\to\infty$ if $\log q+d\log(1-1/q)/2<0$.
A celebrated contribution of Achlioptas and Naor~\cite{AchNaor} shows that for $\GG=\GG(n,d/n)$, 
\begin{align}\label{eqSecondMmt}
\chi(\GG)&\leq q\quad\mbox{\whp}&&\mbox{ if }\qquad d< 2(q-1)\log(q-1).
\end{align}
The proof hinges on the computation of the second moment of the number of $q$-colourings, which involves a delicate analytical optimisation task.
Following up on work of Achlioptas and Moore~\cite{AchMoore}, 
Kemkes, P\'erez-Gim\'enez and Wormald~\cite{KPGW} %DKKKPW} 
showed that \eqref{eqSecondMmt} holds for the random regular graph $\GG=\GG(n,d)$ as well.
Expanding \eqref{eqFirstMmt}--\eqref{eqSecondMmt} asymptotically for large $q$, we find
$\chi(\GG)>q$ if $d>(2q-1)\log q+o_q(1)$, while 
$\chi(\GG)\leq q$ if $d\leq (2q-2)\log q-2+o_q(1)$, with $o_q(1)$ vanishing as $q\to\infty$.
A series of papers~\cite{ACOcovers,ColReg,ACOVilenchik} improved these asymptotic bounds to
\begin{align}\label{eqEnhancedSmm}
\chi(\GG)\begin{cases}\leq q&\mbox{ if }d\leq (2q-1)\log q-2\log 2+o_q(1),\\
>q&\mbox{ if }d>(2q-1)\log q-1+o_q(1)
\end{cases}\qquad\mbox{\whp}
\end{align}
for both the binomial and the random regular graph.
But in the absence of explicit estimates of the $o_q(1)$ error term \eqref{eqEnhancedSmm} fails to render improved bounds for any specific value of $q$.
Finally, several articles have been dedicated to the special case $q=3$.
For the binomial random graph the best bounds read \cite{AMo,DuboisMandler}
\begin{align}\label{eq3}
\chi(\GG(n,d/n))&
\begin{cases}
\leq3&\mbox{if }d\leq 4.03\\
>3&\mbox{if }d>4.94
\end{cases}\qquad\mbox{\whp}
\end{align}
For the random regular graph 
Diaz, Kaporis, Kemkes, Kirousis, P\'erez and Wormald~\cite{DKKKPW}
showed that $\chi(\GG(n,5))=3$ \whp\ if a certain optimisation problem attains its maximum at a specific point, for which they provided numerical evidence.
Moreover, Shi and Wormald~\cite{Shi1,Shi2} proved analytically that  $\chi(\GG(n,4))=3$, while~\eqref{eqFirstMmt} implies that $\chi(\GG(n,6))>3$.
The proofs of all of the above lower bounds rely upon the first moment method, in some cases applied to cleverly designed random variables~\cite{ACOcovers,ColReg,DuboisMandler}.
Similarly, all of the upper bounds derive from second moment arguments, with the exception of the 
upper bound from~\eqref{eq3} and~\cite{Shi1,Shi2}, which are algorithmic. 

Additionally, physicists  brought to bear a canny but non-rigorous technique called the `1RSB cavity method' on the random graph colouring problem~\cite{ZK}.
In the case of the random regular graph, the physics calculations predict an elegant formula. Let
\begin{align}\label{eqSigma}
\Sigma_{d,q}(\alpha)&=\log\bc{\sum_{i=0}^{q-1}(-1)^i\bink q{i+1}(1-(i+1)(1-\alpha)/q)^d}-\frac d2\log(1-(1-\alpha)^2/q) &(\alpha\in[0,1])
\end{align}
and let $\alpha^*\in[0,1]$ be the solution to the algebraic equation
\begin{align}\label{eqalpha}
\sum_{i=0}^{q-1}(-1)^i(1-(i+1)(1-\alpha)/q)^{d-1}\bc{\bink{q-1}{i}-\frac{1-\alpha}q \bink{q}{i+1}}=0
\end{align}
that minimises $\Sigma_{d,q}(\nix)$.   (If there is more than one such value $\alpha^*$, choose one arbitrarily.) Then~\cite{ZK} predicts that
\begin{align}\label{eqreg}
\chi(\GG(n,d))\begin{cases}
\leq q&\mbox{ if } \Sigma_{d,q}(\alpha^*)\geq0\\
>q&\mbox{ if }\Sigma_{d,q}(\alpha^*)<0
\end{cases}\qquad\mbox{\whp}
\end{align}
There is a similarly precise, albeit more complicated, prediction as to the chromatic number of the binomial random graph; see \Sec~\ref{Sec_binom} below.

The aim of this paper is to rigorously establish the lower bounds on the chromatic number predicted by the 
cavity method.
In contrast to prior lower bound arguments, we do not rely on the  first moment method.
Instead, we adapt a technique from the mathematical physics of spin glasses known as the `interpolation method'~\cite{FranzLeone,Guerra,PanchenkoTalagrand} to the graph colouring problem.
In a combinatorial context, the interpolation method has previously been applied to establish a tight lower bound on the random $k$-SAT threshold~\cite{DSS3}, to the independence number of random graphs~\cite{Lelarge} and other optimisation problems on random (hyper-)graphs~\cite{COP,Panchenko2,SSZ} as well as to estimate the rank of random matrices over finite fields~\cite{CG}.
So it may not be surprising that the method can be made to work for graph colouring.
However, the interpolation method remains relatively unknown in combinatorics, where, we believe, it may potentially improve over first moment bounds in many more applications as well.
We therefore endeavour to explain the method at leisure in combinatorial terms to facilitate future applications of the interpolation method.

We proceed to state the results for the random regular graph precisely, followed by the lower bound on the chromatic number of the binomial random graph.
\Sec~\ref{Sec_related} contains references to related work.
An outline of the proof strategy follows in \Sec~\ref{Sec_overview}.

\subsection{The random regular graph}
Given $d,q\geq3$ and with $\Sigma_{d,q}(\alpha)$ from~\eqref{eqSigma} define
\begin{align}\label{eqmySigma}
\Sigma_{d,q}&=\min_{0\leq\alpha\leq 1}\Sigma_{d,q}(\alpha),&
d_q&=\min\cbc{d\geq3\, :\, \Sigma_{d,q}<0}. 
\end{align}
Then we have the following lower bound on the chromatic number of the random regular graph.

\begin{theorem}\label{Thm_reg}
If $q\geq3$ and $d\geq d_q$ then $\chi(\GG(n,d))>q$ \whp
\end{theorem}

The function $\alpha\mapsto\Sigma_{d,q}(\alpha)$ is differentiable and $\Sigma_{d,q}(1)=0$.
Furthermore, the calculations performed towards~\cite[eq.~(35)]{ZK} show that $\Sigma_{d,q}'(0)<0$.
Hence, whenever the minimum value $\Sigma_{d,q}$ is negative, the minimiser $\alpha$ must be 
a zero of $\Sigma_{d,q}'(\nix)$.
It follows, after some algebra, that the minimiser is a solution to~\eqref{eqalpha}.
Thus, \Thm~\ref{Thm_reg} verifies the lower bound from~\eqref{eqreg}, which~\cite{ZK} conjectures to be tight for all $q\geq3$.

Of course, we can evaluate $\Sigma_{d,q}$ numerically and calculate $d_q$ for any given $q$.
The first few values are displayed in Table~\ref{Tab_reg}.
For those $q$ where $d_q$ is displayed in boldface,
 the new bound strictly improves over the first moment bound~\eqref{eqFirstMmt}; {in the other cases the bounds coincide}.
In addition, Table~\ref{Tab_reg} shows the value $d_{q,\mathrm{smm}}$ up to which~\eqref{eqSecondMmt} implies  that $\chi(\GG(n,d))\leq q$ \whp (here ``smm'' is short for ``second moment method'').

\begin{table}[h!]
\begin{tabular}{|c|c|c|c|c|c|c|c|c|c|c|c|c|c|c|c|c|c|c|}\hline
$q$&3&4&5&6&7&8&9&10&11&12&13&14&15&16&17&18&19&20\\\hline
$d_{q,\textrm{smm}}$&-&6&11&16&21&27&33&39&46&52&59&66&73&81&88&96&104&111\\\hline
$d_q$&6&10&15&20&{\bf 25}&{\bf 31}&{\bf 37}&44&{\bf 50}&{\bf 57}&{\bf 64}&{\bf71}&{\bf78}&86&{\bf93}&{\bf101}&109&117\\\hline
\end{tabular}\\[4mm]
\caption{Bounds on the chromatic number of the random regular graph $\GG(n,d)$ for small $d$.}\label{Tab_reg}
\end{table}

The asymptotic lower bound~\eqref{eqEnhancedSmm} on $\chi(\GG(n,d))$, which was derived in~\cite{ColReg} via an extremely laborious first moment argument, follows from \Thm~\ref{Thm_reg} at the expense of just a brief calculation. (See Section~\ref{Sec_asymptotics}.)
\aco{That said, in the limit of large $q$ we do not improve over \eqref{eqEnhancedSmm}, which is conjectured to be optimal up to the precise value of the $o_q(1)$ error term.}

\begin{corollary}\label{Cor_reg}
If $d>(2q-1)\log q-1+o_q(1)$ then $\chi(\GG(n,d))>q$ \whp
\end{corollary}

\subsection{The binomial random graph}\label{Sec_binom}
Locally the random regular graph is as `deterministic' as it gets: for all but a bounded number of exceptional vertices,
any bounded-depth neighbourhood is just a $d$-regular tree \whp\
By contrast, in the binomial random graph $\GG(n,d/n)$ the neighbourhoods  are random, distributed
as the trees generated by a Galton-Watson process with offspring distribution $\Po(d)$.
The value of the chromatic number predicted by the cavity method mirrors this local non-uniformity.
Indeed, while in the case of random regular graphs we obtained the scalar optimisation problem~\eqref{eqmySigma},
in the binomial case we face an optimisation problem over a probability measure on the unit interval.
To be precise, let $\fp$ be a probability distribution on $[0,1]$.
Moreover, let $(\ALPHA_i)_{i\geq1}$ be a family of independent random variables with distribution $\fp$.
Additionally, let $\vD\sim\Po(d)$ be independent of the $\ALPHA_i$.
Then we define
\begin{align}\label{eqERsigma0}
\Sigma^*_{d,q}(\fp)&=\Erw\brk{\log\left(\sum_{i=0}^{q-1}(-1)^i\bink q{i+1}\prod_{h=1}^{\vD}1-(i+1)(1-\ALPHA_h)/q\right)}
		-\frac d2\, \Erw\brk{\log\left(1-(1-\ALPHA_1)(1-\ALPHA_2)/q\right)}.
\end{align}
Setting
\begin{align}\label{eqERsigma}
\Sigma^*_{d,q}&=\inf_{\fp}\ \Sigma^*_{d,q}(\fp),&d_q^*&=\inf\cbc{d>0 \, :\, \Sigma^*_{d,q}<0},
\end{align}
we obtain the following lower bound on the chromatic number.

\begin{theorem}\label{Thm_ER}
If $q\geq 3$ and $d>d_q^*$ then $\chi(\GG(n,d/n))>q$ \whp
\end{theorem}

\noindent
Zdeborov\'a and Krzakala predict that this bound is tight for all $q\geq3$~\cite{ZK}.

Due to the optimisation over distributions $\fp$, the value $d_q^*$  may be hard to evaluate.
The physics literature relies upon a numerical heuristic called population dynamics~\cite{MM} to tackle such optimisation problems,
but of course there is no general guarantee that the true optimiser will be found.
Yet fortunately \Thm~\ref{Thm_ER} shows that {\em any} distribution $\fp$ yields an upper bound on $d_q^*$, and thus a lower bound on the chromatic number.
In particular, we could try atoms $\fp=\delta_\alpha$ with $\alpha\in[0,1]$.
For instance, we find that $\Sigma_{4.697,3}(\delta_{0.25})<0$, whence $d_3^*\leq4.697$; see Figure~\ref{Fig_3col}.
Even this quick bound significantly improves over the best prior bound~\eqref{eq3} from~\cite{DuboisMandler}, based on a tricky first moment calculation, and comes within a whisker of the value $d_3^*\approx 4.687$ obtained via population dynamics~\cite{ZK}.
In principle, the $4.697$ bound could be sharpened by optimising over distributions with a (small) finite support, but such a calculation seems to require computer assistance.

\begin{figure}[ht!]
\includegraphics[height=6cm]{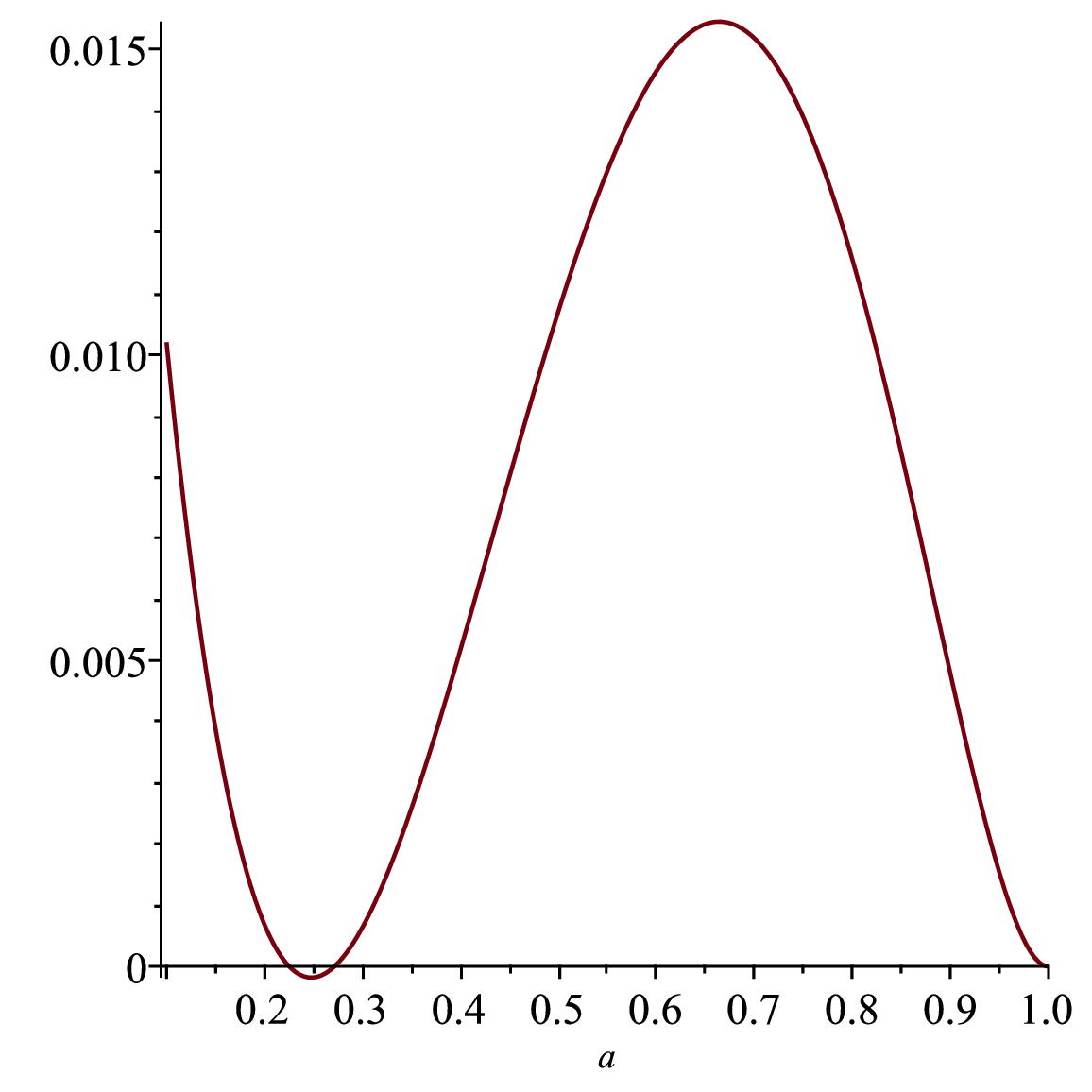}\caption{The function $\Sigma_{4.697,3}(\delta_\alpha)$.}\label{Fig_3col}
\end{figure}

Similarly, substituting a suitable atom $\fp=\delta_\alpha$  into \eqref{eqERsigma} suffices to rederive the large-$q$ asymptotic lower bound on the chromatic number from~\eqref{eqEnhancedSmm}, originally established in~\cite{ACOcovers} via a complicated first moment argument. 
\aco{As in the regular case we do not improve over \eqref{eqEnhancedSmm} asymptotically for large $q$; \csg{again}, \eqref{eqEnhancedSmm} is conjectured to be optimal up to the precise value of the term hidden in the $o_q(1)$.}

\begin{corollary}\label{Cor_ER}
If $d>(2q-1)\log q-1+o_q(1)$ then $\chi(\GG(n,d/n))>q$ \whp
\end{corollary}

\subsection{Related work}\label{Sec_related}
The history of the random graph colouring problem is long and distinguished.
Improving a prior result of Matula~\cite{Matula}, \Bollobas~\cite{BBColor} determined the chromatic number of the dense binomial random $\GG(n,p)$ for fixed $p\in(0,1)$, up to a multiplicative error of $1+o(1)$.
\Kucera\ and Matula obtained the same result via a different proof~\cite{MatulaKucera}.
\Luczak~\cite{LuczakColor} extended the approach from~\cite{Matula,MatulaKucera} to sparse random graphs.
His main result shows that \whp\ for $p=o(1)$,
\begin{align}\label{eqTomasz}
\chi(\GG(n,p))&=\bc{1+O\bcfr{\log\log np}{\log^2 np}}\frac{\log(np)}{2np}.
\end{align}
Particularly for small edge probabilities $p$ the bound~\eqref{eqTomasz} is not quite satisfactory, as a result of Alon and Krivelevich~\cite{AlonKriv} shows that the chromatic number of $\GG(n,p)$ is concentrated on two consecutive integers if $p\leq n^{-1/2-\Omega(1)}$.
Seizing upon techniques from~\cite{AchNaor,AlonKriv}, Coja-Oghlan, Panagiotou and Steger~\cite{Angelika} determined a set of three consecutive integers on which the chromatic number concentrates for $p\leq n^{-3/4-\Omega(1)}$.
Furthermore, the aforementioned result of Achlioptas and Naor~\cite{AchNaor} determines the two integers on which $\chi(\GG(n,d/n))$ concentrates, when $d>0$ is fixed.
Yet in this case it is widely conjectured that there exists a sharp threshold for $q$-colourability, i.e., 
that for each $q\geq3$ there exists $d_q^\star>0$ such that $\chi(\GG(n,d/n))\leq q$ if $d<d_q^\star$ while 
$\chi(\GG(n,d/n))> q$ if $d>d_q^\star$.
Clearly, if such a $d_q^\star$ exists then the chromatic number would actually concentrate on a single integer for almost all $d\in(0,\infty)$.
Towards the sharp threshold conjecture, Achlioptas and Friedgut~\cite{AchFried} established the existence of a non-uniform threshold sequence for every $q\geq3$.
Physics predictions~\cite{ZK} assert that the $q$-colourability threshold $d_q^\star$ coincides with $d_q^*$  from~\eqref{eqERsigma} for all $q\geq3$.

Concerning the random regular graph $\GG(n,d)$, Frieze and \Luczak~\cite{FriezeLuczak} obtained an asymptotic bound akin to~\eqref{eqTomasz} for $d=o(n^{1/3})$, which Cooper, Frieze, Reed and Riordan~\cite{CFRR} extended to $d\leq n^{1-\Omega(1)}$. 
Further, Krivelevich, Sudakov, Vu and Wormald~\cite{KSVW} obtained an asymptotic formula akin to \Bollobas'
result~\cite{BBColor} for degrees $n^{6/7+\Omega(1)}\leq d\leq 0.9n$.
The best prior bounds on $\chi(\GG(n,d))$ with fixed $d$ were stated in \Sec~\ref{Sec_background}.

The physicists' cavity method has inspired a great deal of rigorous work.
Perhaps the most prominent example is the proof of the $k$-SAT threshold conjecture for large $k$ by Ding, Sly and Sun~\cite{DSS3}.
The proof of the lower bound on the $k$-SAT threshold is based on an impressive second moment argument, while the proof of the upper bound relies on the interpolation method.
The way we use the interpolation method here is reminiscent of its application in~\cite{DSS3}.
Further problems in which the 1RSB cavity method has been vindicated include the independent set problem on random regular graphs~\cite{DSS2}, the regular $k$-NAESAT problem~\cite{DSS1} and the regular $k$-SAT problem~\cite{Kosta}.

As for the history of the interpolation method itself, Guerra~\cite{Guerra} invented the technique in order to study the free energy of the Sherrington-Kirkpatrick spin glass model.
The interpolation method went on to become a mainstay of the mathematical physics of spin glasses (see, e.g.,~\cite{Panchenko}).
Franz and Leone~\cite{FranzLeone} pioneered the use of the interpolation method for combinatorial problems.
The approach was further elaborated and generalised by Panchenko and Talagrand~\cite{PanchenkoTalagrand}, and their
version of the interpolation method was applied to the $k$-SAT problem in~\cite{DSS3}.
We will use (and adapt) the Panchenko--Talagrand version as well.
Moreover, an important contribution of Bayati, Gamarnik and Tetali~\cite{bayati} applied a different 
variant of the interpolation method to prove{, e.g.,} the {\em existence} of the 
limit $\lim_{n\to\infty}\alpha(\GG)/n$ of the normalised independence number of the random graph 
$\GG=\GG(n,d)$ or $\GG=\GG(n,d/n)$.
This version of the interpolation method does not provide estimates of the value of such limits.
\aco{Sly, Sun and Zhang~\cite{SSZ}} combined the combinatorial interpolation scheme from~\cite{bayati} with the interpolation arguments from \cite{FranzLeone,PanchenkoTalagrand} to derive bounds on the partition functions of random regular (and uniform) hypergraphs.
For instance, \cite[\Thm~E.3]{SSZ} shows that the formula provided by the 1RSB cavity method yields an upper bound on the partition function for a variety of models.
These models include the Potts antiferromagnet on the random regular graph, which \csg{plays} a prominent role in the present paper as well.
In particular, for the random regular graph $\GG(n,d)$ \Cor~\ref{Prop_PD} below, an important intermediate step towards the proof of \Thm~\ref{Thm_reg},  is a special case of~\cite[\Thm~E.3]{SSZ}.
Furthermore, building upon~\cite{Panchenko2}, Coja-Oghlan and Perkins~\cite{COP} recently used the interpolation method to derive precise variational formulas for
the partition functions of random regular (hyper-)graph models.
The models studied in that paper include the Potts antiferromagnet as well, and the random regular graph version of \Cor~\ref{Prop_PD} could be derived from \cite[Theorem 7.6]{COP} with little effort.
But since the expositions of the 1RSB interpolation method for random regular graphs in~\cite{COP,SSZ} and for binomial random graphs in~\cite{PanchenkoTalagrand} are rather brisk, and since, strictly speaking, \cite{PanchenkoTalagrand} does not cover the Potts model, we present the interpolation method from scratch, with a view to facilitating future uses of the method in combinatorics. 
{
}

\subsection{Preliminaries and notation}
In order to avoid replications and case distinctions, throughout the paper we use the shorthand $\GG$ to denote
either the random regular graph $\GG(n,d)$ or the binomial random graph $\GG(n,d/n)$.
Most of the statements and arguments in the following sections are generic and apply to either model.
\aco{There are just a few steps where we will need to treat the two models separately.}
If $\GG=\GG(n,d/n)$ is the binomial random graph then we let $\vD\disteq\Po(d)$ be a Poisson variable, while
in the case of the random regular graph we let $\vD=d$ deterministically.
In either case we let $(\vD_i)_{i\geq1}$ be independent copies of $\vD$.

As per common practice, we use the $O(\nix)$-notation to refer to the limit $n\to\infty$.
In our calculations we tacitly assume that $n$ is sufficiently large for the various estimates to be valid.
In addition, in \Sec~\ref{Sec_asymptotics} we use $O_q(\nix)$-notation to refer to the limit of large $q$ as in Corollaries~\ref{Cor_reg} and~\ref{Cor_ER}.

For a finite set $\Omega\neq\emptyset$ we denote by $\cP(\Omega)$ the set of probability distributions on $\Omega$.
We identify $\cP(\Omega)$ with the standard simplex in $\RR^{\Omega}$.
Accordingly, $\cP(\Omega)$ inherits its topology from $\RR^{\Omega}$.
Further, we write $\cP^2(\Omega)$ for the space of probability measures on $\cP(\Omega)$.
We endow $\cP^2(\Omega)$ with the weak topology, thus obtaining a Polish space.
Additionally, $\cP^3(\Omega)$ \aco{denotes} the space of probability measures on $\cP^2(\Omega)$.

For a probability measure $\mu$ on a discrete probability space $\cX$ we denote by
$\SIGMA^\mu,\SIGMA^{\mu,1},\SIGMA^{\mu,2},\ldots\in\cX$ independent samples drawn from $\mu$.
Where the reference to $\mu$ is apparent we omit $\mu$ from the superscripts and 
just write $\SIGMA$, $\SIGMA^1$, etc.
For a function $X:\Omega^\ell\to\RR$ we denote the expectation of
$X(\SIGMA^1,\ldots,\SIGMA^\ell)$ by  $\scal X\mu$.
Thus,
\begin{align}\label{eqbck}
\scal X\mu&=\scal{X(\SIGMA^1,\ldots,\SIGMA^\ell)}\mu=\sum_{\sigma^1,\ldots,\sigma^\ell\in\cX}
		X(\sigma^1,\ldots,\sigma^\ell)\prod_{i=1}^\ell\mu(\sigma^i).
\end{align}

Finally, we need the following version of a Markov random field.
A {\em factor graph} 
\[ \cG=(V,\, C,\, (\Omega_v)_{v\in V},\, (\partial a)_{a\in C},\, (\psi_a)_{a\in C})\]
 consists of
\begin{itemize}
\item a finite set $V$ of {\em variable nodes},
\item a finite set $C$ of {\em constraint nodes},
\item a finite or countable range $\Omega_v$ for each $v\in V$,
\item a subset $\partial a\subset V$ for each $a\in C$,
\item a {\em weight function} $\psi_a:\prod_{v\in\partial a}\Omega_v\to[0,\infty)$ for each $a\in C$.
\end{itemize}
A factor graph can be represented by a bipartite graph with vertex sets $V$ and $C$ where the neighbourhood of
$a\in C$ is just $\partial a$.
We further define the function $\psi_{\cG}:\prod_{v\in V}\Omega_v\to\aco{[0,\infty)}$ by 
\begin{equation}
\label{psiG-def}
 \sigma\mapsto\prod_{a\in C}\psi_a(\sigma_{\partial a})
\end{equation}
for all $\sigma = (\sigma_v)_{v\in V}\in \prod_{v\in V}\Omega_v$,
where $\sigma_{\partial a}$ denotes the restriction of $\sigma$ to $\partial a$.
Finally, the \emph{partition function} $Z(\cG)$ of $G$ is defined by
\begin{align}\label{eqZGeneral}
%\psi_{\cG}&:\prod_{v\in V}\Omega_v\to(0,\infty),\qquad
%\sigma=(\sigma_v)_{v\in V}\mapsto\prod_{a\in C}\psi_a(\sigma_{\partial a}),&
Z(\cG)&=\sum_{\sigma\in \prod_{v}\Omega_v}\psi_{\cG}(\sigma).
\end{align}
%and refer to $Z(\cG)$ as the {\em partition function} of $\cG$.
If $0<Z(\cG)<\infty$ then $\cG$ gives rise to a probability distribution
\begin{align}\label{eqBoltzmannGeneral}
\mu_{\cG}(\sigma)&=\psi_{\cG}(\sigma)/Z(\cG)&&\mbox{ for }\sigma\in \prod_{v\in V}\Omega_v
\end{align}
that is called the {\em Boltzmann distribution} of $\cG$.

\section{Outline}\label{Sec_overview}

\noindent
We proceed to survey the proofs of the main results, deferring most technical details to the following sections.

\subsection{The Potts antiferromagnet}\label{Sec_The_Potts_antiferromagnet}
The goal is to derive a lower bound on the chromatic number of the random graph $\GG=\GG(n,d)$ or $\GG=\GG(n,d/n)$.
We tackle this problem indirectly by way of a weighted version of the $q$-colourability problem.
To be precise, the {\em $q$-spin Potts antiferromagnet at inverse temperature $\beta>0$} on a multigraph $G=(V,E)$ is the probability distribution $\mu_{G,\beta}$ on $[q]^V$ defined by
\begin{align}\label{eqPotts1}
\mu_{G,\beta}(\sigma)&=\frac1{Z_\beta(G)}\prod_{vw\in E(G)}1-(1-\eul^{-\beta})\vecone\{\sigma(v)=\sigma(w)\},&&\mbox{where}\\
Z_\beta(G)&=\sum_{\sigma\in[q]^V}\prod_{vw\in E(G)}1-(1-\eul^{-\beta})\vecone\{\sigma(v)=\sigma(w)\}.\label{eqPotts2}
\end{align}
Here it is understood that each edge of $G$ contributes to the products in~\eqref{eqPotts1} and \eqref{eqPotts2} according to its multiplicity.
The strictly positive quantity $Z_\beta(G)$, known as the {partition function}, ensures that $\mu_{G,\beta}$ is a probability measure.
Moreover, we observe that the probability mass $\mu_{G,\beta}(\sigma)$ is governed by the number of edges that $\sigma$ renders monochromatic.
Indeed, the product in \eqref{eqPotts1} imposes an $\exp(-\beta)$ `penalty factor' for every monochromatic edge.
Thus, larger values of $\beta$ deliver higher penalties to monochromatic edges.
In particular, if $\sigma$ is a $q$-colouring of $G$ then the product evaluates to one.
Therefore, the partition function is lower-bounded by the total number of $q$-colourings of $G$
and $\lim_{\beta\to\infty}Z_\beta(G)$ equals the number of $q$-colourings.
Hence, $\chi(G)>q$ if there exists $\beta>0$ such that $Z_\beta(G)<1$.

Thus, our approach is to show that there exists $\beta>0$ such that
if $d$ exceeds the thresholds stated in \Thm s~\ref{Thm_reg} and~\ref{Thm_ER} then \whp\
 $\log Z_\beta(\GG)<0$.
To facilitate the analysis of $Z_\beta$ we will work with slightly modified and (for our purposes)  more amenable random graph models.
Specifically, fixing $\eps>0$, we let
\begin{equation}\label{eqm}
\vm\sim\Po_{\leq dn/2}((1-\eps)dn/2)
\end{equation}
be a Poisson variable conditioned on not exceeding $dn/2$.
Define $\G(n,d/n)$ as the random multigraph on the vertex set $V_n=\{v_1,\ldots,v_n\}$ obtained by 
inserting $\vm$ independent random edges $e_1,\ldots,e_{\vm}$ chosen uniformly out of all $\bink n2$ possible edges.
Similarly, let $\G(n,d)$ be the random multigraph obtained from the following version of the configuration model:
choose a matching $\vec\Gamma$ of size $\vm$ of the complete graph on $V_n\times[d]$ uniformly at random.
Then obtain $\G(n,d)$ by inserting  one $vw$-edge for every matching edge $\{(v,i),(w,j)\}\in\vec\Gamma$.
In order to avoid case distinctions, we use the symbol $\G$ to denote either $\G(n,d/n)$ or $\G(n,d)$.

Working with the Potts antiferromagnet rather than directly with the graph colouring problem offers two advantages.
First, the partition function $Z_\beta(G)$ is always positive and $\log Z_\beta(G)$ enjoys a Lipschitz property with respect to edge additions/deletions.
Indeed, adding or deleting a single edge can change $\log Z_\beta(G)$ by an additive term of
at most $\beta$ in absolute value. (See \aco{\Sec~\ref{eqLip}} below.)
Second, as a consequence of this Lipschitz property it is easy to prove that $\log Z_\beta(\GG)$ is tightly concentrated about its expectation.
Although similar statements already appear in the literature (e.g.,~\cite{Cond,COP}), we include the proof for completeness.

\begin{proposition}\label{Lemma_Azuma}
	For any $\eps,\delta,\beta>0$  there is $\xi>0$ such that \aco{for sufficiently large $n$ we have}
\begin{align}\label{eqLemma_Azuma}
\pr\brk{\abs{\log Z_\beta(\GG)-\Erw\brk{\log Z_\beta(\GG)}}>\delta n}&\leq\exp(-\xi n),&
\pr\brk{\abs{\log Z_\beta(\G)-\Erw\brk{\log Z_\beta(\G)}}>\delta n}&\leq\exp(-\xi n).
\end{align}
\end{proposition}

\noindent
\Prop~\ref{Lemma_Azuma} implies that the partition functions of $\GG$ and $\G$ do not differ too much.

\begin{corollary}\label{Cor_Azuma}
For any $\beta>0$  we have
$\limsup_{\eps\to0}\limsup_{n\to\infty}\frac1n\abs{\Erw\brk{\log Z_\beta(\GG)}-\Erw\brk{\log Z_\beta(\G)}}=0.$
\end{corollary}

\noindent
Finally, thanks to the following corollary it suffices to bound $\Erw[\log Z_\beta(\GG)]$ to show that  $\GG$ fails to be $q$-chromatic.

\begin{corollary}\label{Cor_q}
If there is $\beta>0$ such that $\limsup_{n\to\infty}\frac1n\Erw\brk{\log Z_\beta(\GG)}<0,$ then
$\chi(\GG)>q$ \whp
\end{corollary}
\begin{proof}
If $\chi(G)\leq q$ then $Z_\beta(G)\geq1$ for all $\beta>0$.
Hence, 
\begin{align}\label{eqCor_q1}
\limsup_{n\to\infty}\pr[\chi(\GG)\leq q]>0&\Rightarrow
		\forall\beta>0:\limsup_{n\to\infty}\pr\brk{\ln Z_\beta(\GG)\geq0}>0.
\end{align}
Now, assume that $\limsup_{n\to\infty}\frac1n\Erw[\log Z_\beta(\GG)]<-\delta<0$ for some $\beta>0$.
Then  \Prop~\ref{Lemma_Azuma} implies that 
$\ln Z_\beta(\GG)\leq-\delta n/2$ \whp, and thus
$\limsup_{n\to\infty}\pr\brk{\ln Z_\beta(\GG)\geq0}=0$.
Thus, \eqref{eqCor_q1} shows that  $\chi(\GG)>q$ \whp
\end{proof}

\noindent
The proofs of \Prop~\ref{Lemma_Azuma} and \Cor~\ref{Cor_Azuma} can be found in \Sec~\ref{Sec_conc}. \aco{%and 
\csg{At the end} of \Sec~\ref{Sec_overview} we show} how these results are used to prove our main theorems.

\subsection{The interpolation scheme}
The study of the partition function $Z_\beta(\GG)$ is closely intertwined with the study of the probability distribution $\mu_{\GG,\beta}$ from~\eqref{eqPotts1}.
What turns the latter task into a challenge is the possible presence of extensive stochastic dependencies amongst the colours that $\SIGMA\in[q]^{V_n}$, drawn from $\mu_{\G,\beta}$, assigns to the different vertices.
While there are short range dependencies between the colour of a vertex $v$ and the colours of vertices in its proximity,
the expansion properties of $\GG$ are apt to cause long-range dependencies as well.

To cope with this issue, we are going to compare $\GG$ with another random graph model $\G_1$ in which the dependencies between the vertices are more manageable.
\aco{Specifically, we will upper-bound $\Erw[\log Z_\beta(\GG)]$ in terms of $\Erw[\log Z_\beta(\G_1)]$.}
To this end we will construct an interpolating family of random graphs $(\G_t)_{t\in[0,1]}$ such that $\G_0$ essentially coincides with the random graph $\G$ from \Sec~\ref{Sec_The_Potts_antiferromagnet}.
To compare $\G_0$ and $\G_1$ we will show that $\frac{\partial}{\partial t}\Erw[\log Z_\beta(\G_t)]$ is non-negative.
This general proof strategy is known as the interpolation method.
The specific interpolation scheme $(\G_t)_{t\in[0,1]}$ that we use is an adaptation of the construction that Panchenko and Talagrand~\cite{PanchenkoTalagrand} used to study binary problems on binomial random hypergraphs (e.g., random $k$-SAT formulas).
In the case of random regular graphs, the present construction can actually be viewed as a special case of the interpolation scheme from~\cite{COP}.
But since we need to perform the analysis for the binomial random graph anyway, a unified treatment of both models incurs little overhead.

The elements $\G_t$ of the interpolation scheme will not be plain random graphs but random factor graphs.
To construct the interpolating family, fix a probability measure $\fr\in\cP^3([q])$ as well as parameters $\eps,\beta>0$ and a probability distribution $\gamma$ on $\NN$.
\aco{Let $(\vr_i,\vr_{i,j},\vr_{i}',\vr_i'',\hat\vr)_{i,j\geq1}$ be mutually independent random variables with distribution $\fr$; thus, 
$\vr_i,\vr_{i,j},\vr_{i}',\vr_i'',\hat\vr\in \cP^2([q])$.
Next, given $(\vr_i,\vr_{i,j},\vr_{i}',\vr_i'',\hat\vr)_{i,j}$ let 
\begin{align}\label{eqAllMyRhos}
\{\RHO_{h,i},\, \RHO_{h,i,j},\, \RHO_{h,i}',\, \RHO_{h,i}'',\,\hat\RHO_h \, \mid \, i,j,h\geq 1\}
\end{align}
%\[ \{\RHO_{h,i},\, \RHO_{h,i,j},\, \RHO_{h,i}',\, \RHO_{h,i}'',\,\hat\RHO_h \, \mid \, i,j,h\geq 1\}\] 
be a set of mutually independent random variables such that the $\RHO_{h,i}$ have distribution $\vr_i$, the $\RHO_{h,i,j}$ have distribution $\vr_{i,j}$, the $\RHO_{h,i}'$ have distribution $\vr_i'$, the $\RHO_{h,i}''$ have distribution $\vr_i''$ and the $\hat\RHO_h$ have distribution $\hat\vr$.
Thus, all random variables in
$\{\RHO_{h,i},\, \RHO_{h,i,j},\, \RHO_{h,i}',\, \RHO_{h,i}'',\,\hat\RHO_h \, \mid \, i,j,h\geq 1\}$ 
are mutually independent given $(\vr_i,\vr_{i,j},\vr_{i}',\vr_i'',\hat\vr)_{i,j\geq 1}$.}
Additionally, let
\begin{align}\label{eqPoissons}
\vM_t&\disteq\Po((1-\eps)(1-t)dn/2),&\vM_t'&\disteq\Po((1-\eps)tdn),&\vM_t''&\disteq\Po((1-\eps)(1-t)dn/2)
\end{align}
be mutually independent and independent of everything else.
Define the event
\begin{align}\label{eqM}
	\cM=\cbc{2\vM_t+\vM_t'\leq dn,\,\vM_t+\vM_t'+\vM_t''\leq dn}
\end{align}
and write $(\vm_t,\vm_t',\vm_t'')$ for $(\vM_t,\vM_t',\vM_t'')$ given $\cM$.

\begin{remark}\upshape
\aco{Although the above description of the random variables is complete and correct, now seems to be a propitious moment to dwell on the measure-theoretic basis of the construction. 
It can be implemented on a standard Borel space.
To this end we identify the space $\cP([q])$ with the standard simplex in $\RR^q$.
Thus, $\cP([q])$ inherits the Euclidean topology and the corresponding Borel algebra.
Let $$\fR:[0,1]^2\to\cP([q]),\qquad(x,s)\mapsto\fR_{x,s}$$
be a measurable function and let
$$(\vx_i,\vx_{i,j},\vx_i',\vx_i'',\hat\vx,\vy_{h,i},\vy_{h,i,j},\vy_{h,i}',\vy_{h,i}'',\hat\vy_h)_{h,i,j\geq1}$$
 be mutually independent random variables that are uniformly distributed on the unit interval $[0,1]$, all defined on a common standard Borel space.
Then $\fR$ induces a distribution $\fr\in\cP^3([q])$ as for a given $x\in[0,1]$ we naturally obtain a distribution $\fR_x\in \cP^2([q])$, namely the distribution of the $\cP([q])$-valued random variable $\fR_{x,\vy_{1,1}}$.
Consequently, the distribution $\fr$ of the $\cP^2([q])$-valued random variable $\fR_{\vx_1}$ belongs to the space $\cP^3([q])$.
Indeed, since $\cP([q])$ is a complete separable metric space, any distribution $\fr\in\cP^3([q])$ can be represented by a map $\fR$ in this manner.
Now, the above $\vr_i$ can be identified with the $\cP^2([q])$-valued random variables $\fR_{\vx_i}$, and similarly for $\vr_{i,j},\vr_i',\vr_i'',\hat\vr_i$.
Moreover, the $\RHO_{h,i},\RHO_{h,i,j},\RHO_{h,i}',\RHO_{h,i}'',\hat\RHO_h$ can be identified with $\fR_{\vx_i,\vy_{h,i}},\fR_{\vx_{i,j},\vy_{h,i,j}},\fR_{\vx_i',\vy_{h,i}'},\fR_{\vx_i'',\vy_{h,i}''},\fR_{\hat\vx,\hat\vy_{h}}.$
}\end{remark}

All the factor graphs $\G_t$ have variable nodes
$$s,v_1,\ldots,v_n,$$
with $s$ ranging over $\NN$ (that is, $\Omega_s = \NN$), and $v_1,\ldots,v_n$ ranging over $[q]$.
The constraint nodes are $$e_1,\ldots,e_{\vm_t},a_1,\ldots,a_{\vm_t'},b_1,\ldots,b_{\vm_t''},g.$$
How constraint and variable nodes are connected depends on whether $\GG$ is the binomial or the regular random  graph.

\begin{definition}[binomial case]\label{Def_bin}
The connections between the constraint and variable are as follows.
\begin{itemize}
\item Each $e_i$, $i\in[\vm_t]$, is adjacent to a random pair of two distinct variable nodes from $V_n$; these pairs are drawn uniformly and independently of everything else.
\item Each $a_i$, $i\in[\vm_t']$, is adjacent to $s$ and one random variable node from $V_n$ drawn uniformly and independently of everything else.
\item The constraint nodes $g,b_1,\ldots,b_{\vm_t''}$ are adjacent to the variable node $s$ only.
\end{itemize}
\end{definition}

\noindent
The construction in the random regular case resembles the `configuration model'.

\begin{definition}[regular case]\label{Def_reg}
Let $\vec\Gamma_t$ be a uniformly random maximal matching of the complete bipartite graph with vertex classes
$$
\bc{\bigcup_{i=1}^{\vm_t}\{e_i\}\times\{1,2\}}\cup\bigcup_{i=1}^{\vm'_t}\{a_i\}\qquad\mbox{and}\qquad\bigcup_{i=1}^n\{v_i\}\times[d];
$$
this matching covers the left vertex set completely because $2\vm_t+\vm_t'\leq dn$.
\begin{itemize}
\item Each constraint node $e_i$ is adjacent to the variable nodes $v,w$ for which $\vec\Gamma_t$ contains edges between $(e_i,1)$ and $\{v\}\times[d]$ and $(e_i,2)$ and $\{w\}\times[d]$.
\item Each $a_i$ is adjacent to $s$ and to the variable node $w$ for which $\vec\Gamma_t$ contains an edge between $a_i$ and $\cbc w\times[d]$.
\item The constraints $g,b_1,\ldots,b_{\vm_t''}$ are adjacent to $s$ only.
\end{itemize}
\end{definition}

\noindent
Finally, we need to define the weight functions of the constraint nodes: let
\begin{align*}
\psi_{g}(\sigma_s)&=\gamma(\sigma_s)	&&(\sigma_s\in\NN),\\
\psi_{e_i}(\sigma_v,\sigma_w) &= 1-(1-\eul^{-\beta})\vecone\{\sigma_v=\sigma_w\}
		&& (\partial e_i=\{v,w\},\,\sigma_v,\sigma_{w}\in[q]),\\
\psi_{a_i}(\sigma_s,\sigma_v)&=1-(1-\eul^{-\beta})\RHO_{\sigma_s,i}(\sigma_v)
		&&(\partial a_i=\{s,v\},\,\sigma_s\in\NN,\ \sigma_v\in[q]),\\
\psi_{b_i}(\sigma_s)&=1-(1-\eul^{-\beta})\sum_{\tau\in[q]}\RHO_{\sigma_s,i}'(\tau)\RHO_{\sigma_s,i}''(\tau)
	&&(\sigma_s\in\NN).
\end{align*}
Thus, $\psi_{g}$ simply weighs the value $s$ according to the given probability distribution $\gamma$.
Moreover, the constraint nodes $e_i$ simulate the effect of the edges of the original graph $\G$ as in the definition~\eqref{eqPotts1} of the Potts model.
Indeed, if the variable nodes adjacent to $e_i$ are coloured the same then the weight is $\exp(-\beta)$; otherwise it is one.
Moreover, $\psi_{a_i}$ weighs the colour $\sigma$ of the adjacent variable node from $V_n$ according to $\RHO_{\sigma_s,i}$.
Further, $\psi_{b_i}(\sigma_s)$ is determined by the probability that two colours chosen independently  from
$\RHO_{\sigma_s,i}',\RHO_{\sigma_s,i}''\in\cP([q])$ coincide.
The total weight $\psi_{\G_t}$, partition function $Z(\G_t)$ and the Boltzmann distribution $\mu_{\G_t}$ are defined by the general formulas \eqref{psiG-def}--\eqref{eqBoltzmannGeneral}.
In the physics literature the $a_i$ are called \emph{external fields}~\cite{MM}.
A similar construction involving an extra $\NN$-valued variable node $s$ was used in~\cite{SSZ}.

At `time' $t=1$ \eqref{eqPoissons} ensures that $\vm_t=\vm_t''=0$.
Thus, the only constraints present are the $a_i$.
Each of them is connected to the variable node $s$ and to one other variable node.
Hence, the factor graph is star-shaped with constraint node $s$ at the centre.
In effect, the variable nodes $v_1,\ldots,v_n$ are dependent only through $s$.

By contrast, at $t=0$ \eqref{eqPoissons} yields $\vm_t'=0$.
Thus, the factor graph contains only constraints of type $e_i$ and of type $b_i$.
In effect, $\G_0$ decomposes into two parts.
The connected component of $s$ contains all the constraint nodes $b_i$, none of which is connected with $v_1,\ldots,v_n$.
Thus, once more there is a star structure with $s$ at the centre, and it is not too difficult to write out the partition function of this component.
Furthermore, the factor graph induced on $v_1,\ldots,v_n$ and $e_1,\ldots,e_{\vm_1}$ is essentially identical to the original graph $\G$.
\aco{More specifically, the Boltzmann distribution \aco{$\mu_{\G_0}$} mimics that of the Potts antiferromagnet $\mu_{\G,\beta}$ from~\eqref{eqPotts1}. The only, for our purposes negligible, difference is that $\G_0$ typically has slightly fewer than $dn/2$ constraint nodes of the type $e_i$.}
%  Viktor suggested deleting ``hope to''
Thus, \csg{we can %hope to 
relate} the partition functions $Z_\beta(\G)$ and $Z(\G_0)$; see Figure~\ref{Fig_01} for an illustration.

%\begin{figure}[ht!]
%\includegraphics[height=4cm]{colint1.pdf}\hfill\includegraphics[height=4cm]{colint2.pdf}
%\end{figure}
\begin{figure}[ht!]
\begin{center}
\begin{tikzpicture}%[scale=0.9]
\draw [-] (1,0)--(1,2)--(2.5,0) -- (5,2)--(7,0)--(7,2) -- (4,0) -- (3,2) -- (1,0);
\draw [-] (1,2)--(5.5,0) -- (3,2);
\draw [fill=white] (0.75,-0.25) rectangle (1.25,0.25); \draw [fill=white] (2.25,-0.25) rectangle (2.75,0.25);
\draw [fill=white] (3.75,-0.25) rectangle (4.25,0.25); \draw [fill=white] (5.25,-0.25) rectangle (5.75,0.25);
\draw [fill=white] (6.75,-0.25) rectangle (7.25,0.25); \draw [fill=white] (1,2) circle (0.25);
\draw [fill=white] (3,2) circle (0.25); \draw [fill=white] (5,2) circle (0.25);
\draw [fill=white] (7,2) circle (0.25);
\node at (1,0) {$e_1$}; \node at (2.5,0) {$e_2$}; \node at (4,0) {$e_3$}; \node at (5.5,0) {$e_4$};
\node at (7,0) {$e_5$}; \node at (1,2) {$v_1$}; \node at (3,2) {$v_2$}; \node at (5,2) {$v_3$};
\node at (7,2) {$v_4$}; 
%%%
\draw [-] (2.5,-2) -- (4,-1) -- (4,-2)  (2.5,-1)--(4,-1) -- (5.5,-2);
\draw [fill=white] (4,-1) circle (0.25);
\draw [fill=white] (2.25,-2.25) rectangle (2.75,-1.75); \draw [fill=white] (3.75,-2.25) rectangle (4.25,-1.75);
\draw [fill=white] (5.25,-2.25) rectangle (5.75,-1.75); \draw [fill=white] (2.25,-1.25) rectangle (2.75,-0.75);
\node at (4,-1) {$s$}; \node at (2.5,-2) {$b_1$}; \node at (4,-2) {$b_2$}; \node at (5.5,-2) {$b_3$};
\node at (2.5,-1) {$g$}; 
\begin{scope}[shift={(8,0)}]
\draw [-] (1.25,0) -- (2,2) -- (2,0) (2,2)--(2.75,0) (3.5,0) -- (4.25,2) -- (4.25,0) (4.25,2) -- (5,0);
\draw [-] (5.75,0) -- (6.5,2) -- (6.5,0) (6.5,2) -- (7.25,0) (8.0,0) -- (8.375,2) -- (8.75,0); 
\draw [-] (1.25,-0.2) -- (5,-2) -- (2,-0.2) (2.75,-0.2) -- (5,-2)-- (3.5,-0.2) (4.25,-0.2) -- (5,-2) -- (5,-0.2);
\draw [-] (5.75,-0.2) -- (5,-2) -- (6.5,-0.2) (7.25,-0.2) -- (5,-2)-- (8,-0.2) (8.8,-0.2) -- (5,-2) -- (3,-2);
\draw [fill=white] (2,2) circle (0.25); \draw [fill=white] (4.25,2) circle (0.25);
\draw [fill=white] (6.5,2) circle (0.25); \draw [fill=white] (8.375,2) circle (0.25);
\draw [fill=white] (5,-2) circle (0.25);
\draw [fill=white] (1.00,-0.25) rectangle (1.5,0.25); \draw [fill=white] (1.75,-0.25) rectangle (2.25,0.25);
\draw [fill=white] (2.50,-0.25) rectangle (3,0.25); \draw [fill=white] (3.25,-0.25) rectangle (3.75,0.25);
\draw [fill=white] (4.00,-0.25) rectangle (4.5,0.25); \draw [fill=white] (4.75,-0.25) rectangle (5.25,0.25);
\draw [fill=white] (5.50,-0.25) rectangle (6,0.25); \draw [fill=white] (6.25,-0.25) rectangle (6.75,0.25);
\draw [fill=white] (7.00,-0.25) rectangle (7.5,0.25); \draw [fill=white] (7.75,-0.25) rectangle (8.25,0.25);
\draw [fill=white] (8.50,-0.25) rectangle (9.00,0.25); \draw [fill=white] (2.75,-2.25) rectangle (3.25,-1.75); % g
\node at (1.25,0) {$a_1$}; \node at (2,0) {$a_2$}; \node at (2.75,0) {$a_3$}; \node at (3.5,0) {$a_4$};
\node at (4.25,0) {$a_5$}; \node at (5,0) {$a_6$}; \node at (5.75,0) {$a_7$}; \node at (6.5,0) {$a_8$};
\node at (7.25,0) {$a_9$}; \node at (8,0) {$a_{10}$}; \node at (8.75,0) {$a_{11}$}; \node at (5,-2) {$s$};
\node at (2,2) {$v_1$}; \node at (4.25,2) {$v_2$}; \node at (6.5,2) {$v_3$}; \node at (8.375,2) {$v_4$};
\node at (3,-2) {$g$};
\end{scope}
\end{tikzpicture}
\end{center}
\caption{The factor graphs $\G_0$ (left) and $\G_1$ (right).}\label{Fig_01}
\end{figure}
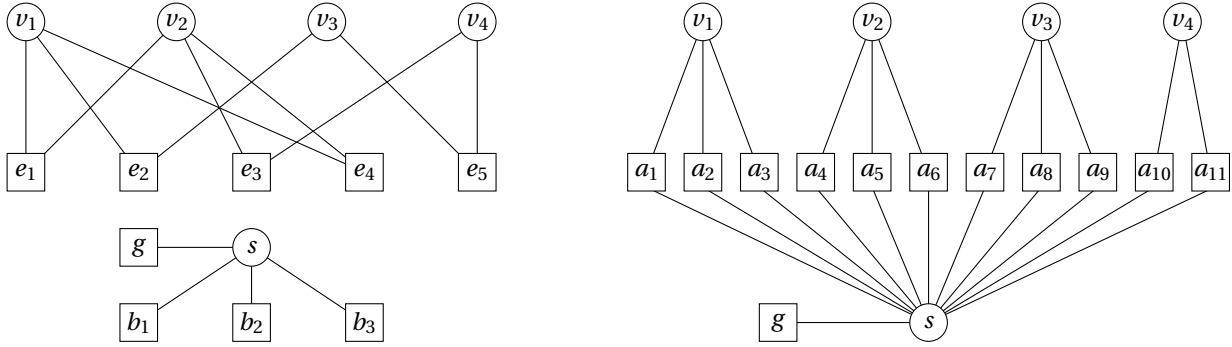

We observe that the distribution of the degrees of $v_1,\ldots,v_n$ remains essentially the same for $0\leq t\leq1$.
Specifically, in the regular case most variables have degree exactly $d$ throughout the interpolation, and in the binomial case the degrees are approximately  $\Po((1-\eps)d)$ distributed.
Additionally, the total number of constraints remains (essentially) constant throughout the interpolation as well.
Indeed, at $t=0$ there are about $(1-\eps)dn/2$ constraints of type $e_i$ and about the same number of constraints $b_i$, while
at $t=1$ we have about $(1-\eps)dn$ constraints of type $a_i$.

As mentioned above, the idea behind the construction is to compare $\Erw[\log Z_\beta(\GG)]$ with the partition function of a simpler model where correlations amongst $v_1,\ldots,v_n$ are amenable to a precise analysis.
The following two propositions spell out this relationship precisely.

%\csg{[[CSG: In the previous draft we had $\delta n$ and $o(n)$ in the lower bound below, now just $\delta n$.]]}

\begin{proposition}\label{Prop_inter1}
Let
\begin{align*}
Y'&=\sum_{\sigma_s=1}^\infty \gamma(\sigma_s)\, \prod_{1\leq i\leq dn/2} 1-(1-\eul^{-\beta})\,
  \sum_{\tau=1}^q\,
    \RHO_{\sigma_s,i}'(\tau)\, \RHO_{\sigma_s,i}''(\tau).
\end{align*}
\aco{Then for any $\delta, \beta>0$ there exists $\eps_0(d,\delta,\beta)>0$ such that for all $0<\eps<\eps_0(d,\delta,\beta)$ and for all $n>1/\eps_0(d,\delta,\beta)$ we have $\Erw[\log Z(\G_0)]\geq {\Erw[\log Z_\beta(\GG)]}+\Erw[\log Y'] - \delta n$.}
\end{proposition}

\noindent
Furthermore, the following proposition shows that $Z(\G_1)$ dominates $Z(\G_0)$.
The proof is based on estimating the derivative $\frac{\partial}{\partial t}\Erw[\log Z(\G_t)]$.

\begin{proposition}\label{Prop_inter2}
We have\ $\Erw[\log Z(\G_0)]\leq\Erw[\log Z(\G_1)]+o(n).$
\end{proposition}

\noindent
Finally, we introduce a convenient proxy for the partition function of $\G_1$: let
\begin{align}\label{eqY}
Y&=\sum_{\sigma_s=1}^\infty \gamma(\sigma_s)\, \prod_{i=1}^n\, \sum_{\sigma_{v_i}=1}^q
  \, \prod_{j=1}^{\vD_i}\,
  \bc{1-(1-\eul^{-\beta})\RHO_{\sigma_s,i,j}(\sigma_{v_i})}.
\end{align}

\begin{corollary}\label{Prop_inter}
For any $\beta>0$ we have\, ${\Erw[\log Z_\beta(\GG)]}\leq \Erw[\log Y]-\Erw[\log Y'] + o(n)$.
\end{corollary}

\noindent
The proofs of \Prop s~\ref{Prop_inter1} and~\ref{Prop_inter2} and \Cor~\ref{Prop_inter} can be found in \Sec~\ref{Sec_Interpolation}.
We are thus left to study $Y,Y'$, which are approximations to the partition function of the factor graph $\G_1$
and the partition function of the $s$-component of $\G_0$, respectively.

\aco{Let us wrap up by dwelling on the intended combinatorial semantics of construction.
	The nodes $v_1,\ldots,v_n$ and $e_1,\ldots,e_{\vm_t}$ clearly mimic the orignial Potts antiferromagnet.
	But as we move along we replace more and more $e_i$ by external fields $a_j$.
	These are meant to capture the physics intuition as to the nature of the interactions between variable nodes in the random graph $\GG$; \Cor~\ref{Prop_inter} corroborates the physics picture to the extent that it yields an upper bound on the partition function.
	Specifically, the impact that an actual edge $e=vw$ of $\GG$ has on an incident vertex $v$ is thought to be governed by the local graph structure around the other vertex $w$ in the graph $\GG-e$ obtained by removing $e$~\cite{MM}.
	Since short cycles are scarce, the local graph structure will likely be a tree.
	Indeed,  it will just be a $(d-1)$-ary tree in the random regular graph, and a $\Po(d)$ Galton-Watson tree in the binomial case.
%	Hence, the random variables $\vr_i$ will play an important role in the binomial setting, while we will replace them by a deterministic value in the regular case.
	In the binomial case, the specific tree structure is apt to impact the influence that $w$ exerts on $v$.
	For example, if the Galton-Watson tree dies out quickly, then it will be easy to colour the entire tree properly regardless \csg{of} the colour of $v$.
	Thus, the edge $e=vw$ will be of little consequence.
	By contrast, in the event of a relatively dense tree, choosing a specific colour for $v$ might have repercussions on a large number of other vertices.
	The random variables $\vr_i$ are meant to capture the randomness of the tree structure pending on vertex $w$.
	But for the sake of simplicity, we do not %actually bother incorporating 
\csg{incorporate} an actual distribution on trees into our construction.
	Instead, we %just 
\csg{make do} with the distribution $\vr_i\in\cP^2([q])$ that is meant to just capture the ensuing impact that $w$ has on $v$.
}

\aco{
	Furthermore, the variable node $s$ is intended to represent the conjectured structure of the Boltzmann distribution $\mu_{\GG,\beta}$.
	To elaborate: according to physics intuition, the distribution $\mu_{\GG,\beta}$ partitions the phase space $[q]^{V_n}$ into an unbounded number of `clusters' $(S_i)_i$ for $d$ close to the $q$-colourability threshold and $\beta$ large~\cite{MP,ZK}.
	Inside a cluster, i.e., under the conditional distribution $\mu_{\GG,\beta}(\nix\mid S_i)$, most vertices $w$ are strongly polarised towards a particular colour.
	In other words, the conditional marginals $\mu_{\GG,\beta}(\{\SIGMA_w=c\nix\mid S_i)$ for $c\in[q]$ are typically either fairly close to zero or to one, while of course overall the marginal of the colour of each vertex is just uniform.
	The variable node $s$ is intended to represent the choice of the cluster $S_i$.
	Thus, the distribution $\vr_i\in\cP^2([q])$, which mimics the local graph structure, determines how the marginal of $w$ is distributed given a cluster index, and then the sample $\RHO_{\sigma_s,i}(c)$ represents the actual realisation of the distribution of the colour inside cluster number $\sigma_s$.
	Finally, $\gamma(s)$ models the distribution of the relative cluster volumes $\mu_{\GG,\beta}(S_i)$.
%	Finally, the constraint nodes $b_i$ are included to balance the total number of constraints.
%	This is necessary because for each actual pairwise interaction $e_i$ that gets removed during the interpolation, we add two external fiels $a_i$ on the average.
%	Thus, we add a corresponding number of $b_i$ to effectively put $\G_0$ and $G_1$ on the same footing.
}

\subsection{Poisson-Dirichlet weights}\label{Sec_PD}
While the expression $\Erw[\log Y]-\Erw[\log Y']$ from \Cor~\ref{Prop_inter} already bears a certain resemblance to~\eqref{eqERsigma0}, an important difference remains.
Namely, the expressions $Y,Y'$ inside the logarithm still contain $n$, the number of vertices.
If the probability distribution $\gamma$ is an atom, that is, $\gamma(h)=1$ for some $h\in\NN$,
then we can produce the same joint distribution on $\sigma_{v_1},\ldots, \sigma_{v_n}$ by
deleting $s$ and $g$ from the factor graphs $G_0$ and $G_1$ and replacing $\sigma_s$ by $h$
in the expressions for $Y$ and $Y'$.  This causes $Y$ and $Y'$ to factorise:
\begin{align}\label{eqgammaatom}
\Erw[\log Y]&=\sum_{i=1}^n\Erw\brk{\log\sum_{\tau=1}^q\, \prod_{j=1}^{\vD_i}\, 1-(1-\eul^{-\beta})\RHO_{h,i,j}(\tau)},\ 
\Erw[\log Y']= \sum_{i\leq dn/2}\, \Erw\brk{\log \bc{1-(1-\eul^{-\beta})\sum_{\tau=1}^q \, \RHO_{h,i}'(\tau)\RHO_{h,i}''(\tau)}}.
\end{align}
In particular, long-range correlations are completely absent in the target $\G_1$ of the interpolation.
(The modified $\G_1$ with $s$ and $g$ deleted consists of $n$ connected components, each
containing exactly one $v_i$.) 
In physics jargon the bound on $\Erw[\log Z_\beta(\G)]$ that can be obtained from~\eqref{eqgammaatom}  is called the {\em replica symmetric bound}.
While the replica symmetric bound easily implies the first moment bound~\eqref{eqFirstMmt},
it does not appear sufficient to prove \Thm s~\ref{Thm_reg} and~\ref{Thm_ER} for any $q\geq3$.

Fortunately there is another choice of the distribution $\gamma$ that leads to a simple formula.
Recall that the {\em Poisson-Dirichlet distribution} with parameter $y>0$ is defined as follows.
Let $\vec P\subset(0,\infty)$ be the countable point set generated by a Poisson point process on $(0,\infty)$ with density $x^{-1-y}\dd x$, independent of all other sources of randomness that have been introduced thus far.
Further, let $(\vec p_h)_{h\geq1}$ be the sequence that comprises the points from $\vec P$ in decreasing order, i.e., $\vec p_h\geq\vec p_{h+1}$ for all $h$.
Since $y>0$, we have $\sum_{h=1}^{\infty}\vec p_h<\infty$ almost surely.
Therefore,
\begin{align}\label{eqGAMMA}
\GAMMA(h)=\vec p_h\Big/\sum_{\ell=1}^{\infty}\vec p_\ell
\end{align}
defines a probability measure on $\NN$, the Poisson-Dirichlet law.
To be precise, $\GAMMA$ is a random probability measure  which depends on $\vec P$. 
\csg{This distribution is used in the} following lemma, \csg{which} enables us to simplify $\Erw[\log Y],\ \Erw[\log Y']$.

\begin{lemma}[\aco{\cite[\Prop~1]{PanchenkoTalagrand} and~\cite[\Prop~6.5.15]{Talagrand}}]\label{Lemma_PD}
	\aco{Suppose that $0<y<1$ and} that $(X_h)_{h\geq1}$ are positive identically distributed random variables with bounded second moments, mutually independent and independent of\, $\GAMMA$.
Then
\begin{align*}
\Erw\brk{\log\sum_{h\geq1}\GAMMA(h)X_h}&=\frac1y\log\Erw[X_1^y].
\end{align*}
\end{lemma}

\aco{In the physics literature, the Poisson-Dirichlet distribution has been postulated as the correct distribution of the relative cluster sizes~\cite{MM,MP}.}
Moreover, Panchenko and Talagrand~\cite{PanchenkoTalagrand} used \Lem~\ref{Lemma_PD} to bound the partition function of the random $k$-SAT model.
We apply \Lem~\ref{Lemma_PD} in a similar manner to upper bound $\Erw[\log Z_\beta(\GG)]$.
Specifically, let $\cR$ be the $\sigma$-algebra generated by \aco{$(\vr_i,\vr_{i,j},\vr_i',\vr_i'',\vD_i)_{\csg{i,j\geq1}}$}.
%% CSG: reworded the next sentence to avoid an overfull line.
Thanks to \Lem~\ref{Lemma_PD}, \csg{we can simplify} the bound from~\Cor~\ref{Prop_inter}
% simplifies 
as follows.

\begin{corollary}\label{Prop_PD}
For any $y,\beta>0$ and $\fr\in\cP^3([q])$ we have
\begin{align}\label{eqProp_PD}
\limsup_{n\to\infty}&{\frac 1n\Erw[\log Z_\beta(\GG)]}\leq\phi_{\beta,y}(\fr)/y,\qquad\mbox{where}\nonumber\\
\phi_{\beta,y}(\fr)&=
	\Erw\left[\,\, \log\, \Erw\brk{\bc{\sum_{\tau=1}^q\prod_{j=1}^{\vD_1}1-(1-\eul^{-\beta})\RHO_{1,1,j}(\tau)}^y\,\bigg|\,\cR}
	-\frac d{2}\log\, \Erw\brk{\bc{1-(1-\eul^{-\beta})\sum_{\tau=1}^q\RHO_{1,1}'(\tau)\RHO_{1,1}''(\tau)}^y\,\bigg|\,\cR}\,\, \right].
\end{align}
%with the outer\, $\Erw\brk\nix$ being the expectation on $\vr_{1,1},\vr_1',\vr_1'',\vD_1$ only.
\end{corollary}

\begin{proof}
{Choose $\eps>0$ small enough}
and assume that $n$ is sufficiently large.
%\csg{[[CSG: We used to say `` Fix $\delta>0$, choose $\eps=\eps(\delta)>0$ small enough'' but I don't think we have any $\delta$ any more, do we?]]}
Moreover, for all $h\in\NN$ let
\begin{align*}
X_h&=\prod_{i=1}^n\sum_{\sigma_{v_i}=1}^q\prod_{j=1}^{\vD_i}\bc{1-(1-\eul^{-\beta})\RHO_{h,i,j}(\sigma_{v_i})},&
X_h'&= \prod_{1\leq i\leq dn/2}\, \bc{1-(1-\eul^{-\beta})\sum_{\tau=1}^q\RHO_{h,i}'(\tau)\RHO_{h,i}''(\tau)}.
\end{align*}
Applying \Cor~\ref{Prop_inter} to the random distribution $\GAMMA$, we obtain
\begin{align*}
\Erw[\log Z_\beta(\GG)]&\leq \Erw\brk{\log\sum_{h=1}^\infty\GAMMA(h) X_h}-\Erw\brk{\log\sum_{h=1}^\infty\GAMMA(h) X_h'} + o(n).
\end{align*}
Hence, \Lem~\ref{Lemma_PD} yields
\begin{align}\label{eqProp_PD_1}
y\, \Erw[\log Z_\beta(\GG)]&\leq
	\Erw\brk{\log\Erw\brk{X_1^y\mid\cR}-\log\Erw\brk{{X_1'}^y\mid\cR}} + o(yn);
\end{align}
clearly, since $X_1,X_1'$ do not depend on $\vec P$, the outer $\Erw\brk\nix$ in \eqref{eqProp_PD_1} is on $(\vr_i,\vr_{i,j},\vr_i',\vr_i'',\vD_i)_{i\geq1}$ only.
Further, because %the $\vD_i$, 
the $\RHO_{h,i,j},\RHO_{h,i}',\RHO_{h,i}''$ are mutually independent given $\cR$, we obtain
\begin{align}\label{eqProp_PD_2}
\Erw\brk{\log\, \Erw[X_1^y\mid\cR]}&= n\, \Erw\brk{\log\, \Erw\brk{\bc{\sum_{\tau=1}^q\prod_{j=1}^{\vD_1}1-(1-\exp(-\beta))\RHO_{1,1,j}(\tau)}^y\,\bigg|\,\cR}},\\
\Erw\brk{\log\, \Erw[{X_1'}^y\mid\cR]}&=\frac{dn}2\, \Erw\brk{\log\, \aco{\Erw\brk{\bc{1-(1-\exp(-\beta))\sum_{\tau=1}^q\RHO_{1,1}'(\tau)\RHO_{1,1}''(\tau)}^y\,\bigg|\,\cR}}}.\label{eqProp_PD_3}
\end{align}
The assertion follows from \eqref{eqProp_PD_1}--\eqref{eqProp_PD_3}. 
\end{proof}

\subsection{The zero temperature limit}
To actually deduce a bound on the chromatic number from \Prop~\ref{Prop_PD} we need to fix the three remaining parameters $\beta,y,\fr$.
Since the Potts model approaches the graph colouring problem in the limit of large $\beta$, it seems natural to take the limit $\beta\to\infty$. 
In physics jargon, we take the `temperature' $1/\beta$ to zero.
Moreover, physics intuition suggests sending the `Parisi parameter' $y$ to zero {as well}.
Ding et al.~\cite{DSS3} took similar limits %were taken in~\cite{DSS3} 
to derive the upper bound on the $k$-SAT threshold
from the formula for the $k$-SAT partition function from~\cite{PanchenkoTalagrand}.

With respect to $\fr$, we make two different choices, depending on whether $\GG$ is  regular or binomial.
Let us begin with the regular case.
For $i\in[q]$, let $\eta_i\in\cP([q])$ be the atom on colour $i$.
Moreover, let $\eta_0=q^{-1}\vecone\in \cP([q])$ be the uniform distribution on the $q$ colours.
Then for a given $\alpha\in[0,1]$ we define
\begin{align}\label{eqralpha}
r_\alpha&=\alpha\, \delta_{\eta_0}+\frac{1-\alpha}{q}\sum_{i=1}^q\delta_{\eta_i}\in\cP^2([q]).
\end{align}
Geometrically, we can think of $r_\alpha$ as a discrete distribution on the standard simplex $\cP([q])\subset\RR^q$ that places mass $\alpha$ on the centre and distributes the remaining mass $1-\alpha$ equally amongst the $q$ vertices of the simplex.
Let
\begin{align}\label{eqrhoalpha}
\fr_\alpha&=\delta_{r_\alpha}\in\cP^3([q])
\end{align}
be the atom on $r_\alpha$.
Further, the expression~\eqref{eqERsigma} for the binomial random graph involves a probability distribution $\fp$ on $[0,1]$.
Given any $\fp\in \cP([0,1])$, we define 
\begin{align}\label{eqrhoALPHA}
\fr_{\fp}&=\int_0^1\fr_{\alpha}\, \dd\fp(\alpha). %\in\cP^3([q]).
\end{align}
\aco{Observe that the integrand is the distribution $\fr_\alpha\in\cP^3([q])$ from \eqref{eqrhoalpha}, and thus $\fr_{\fp}\in \cP^3([q])$.} 
Plugging $\fr_\alpha$ or $\fr_{\fp}$ into \Prop~\ref{Prop_PD}, we finally obtain the expressions from~\eqref{eqmySigma} and~\eqref{eqERsigma}.

\begin{proposition}\label{Prop_zero}
If $\GG=\GG(n,d)$ is the random regular graph then
\begin{align*}
\lim_{y\to0}\lim_{\beta\to\infty}\phi_{\beta,y}(\fr_\alpha)&=\Sigma_{d,q}(\alpha)&\mbox{for all }\alpha\in[0,1].
\end{align*}
Moreover, if $\GG=\GG(n,d/n)$ is the binomial model then
\begin{align*}
\lim_{y\to0}\lim_{\beta\to\infty}\phi_{\beta,y}(\fr_{\fp})&=\Sigma_{d,q}^*(\fp)&\mbox{for all }\fp\in\cP([0,1]).
\end{align*}
\end{proposition}

\noindent
The proof of \Prop~\ref{Prop_zero} can be found in \Sec~\ref{Sec_Prop_zero}.
\bigskip

\noindent Now we have all the pieces in place to complete the proofs of the main theorems.

\begin{proof}[Proof of \Thm~\ref{Thm_reg}]
Fix $q\geq3$ and assume that $\Sigma_{d,q}<0$ for some $d\geq3$.
(This holds when $d=d_q$, for example.)
Then \Prop~\ref{Prop_zero} yields $y,\beta>0$ and $\alpha\in[0,1]$ such that $\phi_{\beta,y}(\fr_\alpha)<0$.
In particular we can take $\alpha$ to be the value which minimises $\Sigma_{d,q}(\cdot)$.
Consequently, \Cor~\ref{Prop_PD} implies that $\limsup_{n\to\infty}\frac1n\Erw[\log Z_\beta(\GG)]<0$.
Therefore, \Cor~\ref{Cor_q} implies that 
\begin{equation}\label{eqThm_reg}
\chi(\GG(n,d))>q\qquad\mbox{\whp}
\end{equation}

We are left to prove that $\GG(n,d')$ also fails to be $q$-chromatic \whp\ for all $d'>d$.
To see this, we observe that the property of being $q$-colourable is decreasing; that is, if a graph $G$ is $q$-colourable then so is every subgraph $G'$ of $G$.
Now, \cite[\Thm~9.36]{JLR} shows that if $d'>d$ and if a decreasing property $\cA$ is satisfied for $\GG(n,d')$ \aco{\whp\ then} $\GG(n,d)$ enjoys $\cA$ \whp, too.
Thus, \eqref{eqThm_reg} implies that $\chi(\GG(n,d'))>q$ for all $d'>d$.
\end{proof}

\begin{proof}[Proof of \Thm~\ref{Thm_ER}]
Once more we fix $q\geq3$  and suppose that $\Sigma^*_{d,q}<0$ for some $d>0$.
(This holds when $d=d_q^*$, for example.)
Then by \Prop~\ref{Prop_zero} there exist $y,\beta>0$ and $\fp\in\cP([0,1])$ 
such that $\phi_{\beta,y}(\fr_\fp)<0$ and thus
$\limsup_{n\to\infty}\frac1n\Erw[\log Z_\beta(\GG)]<0$ by \Cor~\ref{Prop_PD}.
Hence, \Cor~\ref{Cor_q} yields
\begin{equation}\label{eqThm_ER}
\chi(\GG(n,d/n))>q\qquad\mbox{\whp}
\end{equation}
Finally, due to monotonicity, \eqref{eqThm_ER} implies that $\chi(\GG(n,d'/n))>q$ \whp\ for all $d'>d$.
\end{proof} 

\noindent
Given \Thm s~\ref{Thm_reg} and~\ref{Thm_ER} the asymptotic formulas detailed in \Cor~\ref{Cor_reg} and \Cor~\ref{Cor_ER} follow from routine calculations, which we defer to \Sec~\ref{Sec_asymptotics}.

\section{Concentration}\label{Sec_conc}

\noindent
After proving \Prop~\ref{Lemma_Azuma} in \Sec~\ref{Sec_Lemma_Azuma}, we prove
\Cor~\ref{Cor_Azuma} in \Sec~\ref{Sec_Cor_Azuma}.

\subsection{Proof of \Prop~\ref{Lemma_Azuma}}\label{Sec_Lemma_Azuma}
The proof is based on Azuma's inequality and the Lipschitz property of the random variable $\log Z_\beta(\nix)$.
Indeed, \eqref{eqPotts2} shows that if a multigraph $G'$ is obtained from $G$ by adding one single edge then
$e^{-\beta}\leq Z_\beta(G')/Z_\beta(G) \leq 1$, and hence
\begin{align}\label{eqLip}
\abs{\log Z_\beta(G')-\log Z_\beta(G)}&\leq\beta.
\end{align}
We pick a small enough $\zeta=\zeta(\eps,\delta,\beta)>0$ and a smaller $\xi=\xi(\eps,\delta,\beta,\zeta)>0$.
We treat the binomial random graph and the random regular graph separately, 
tacitly assuming in either case that $n$ is sufficiently large.

\subsubsection{The binomial random graph}
Writing $\vM\disteq\Bin(\binom n2,d/n)$ for the number of edges of $\GG$ and invoking the Chernoff bound, we obtain
\begin{align}\label{eqERedges1}
\pr\brk{\abs{\vM-dn/2}<\zeta n}\geq1-\exp(-4\xi n).
\end{align}
Further, let $\G_{n,m}$ be the random multigraph on $n$ vertices comprising $m$ edges chosen uniformly and independently out of all $\bink n2$ possible edges.
Let $\cS$ be the event that $\G_{n,m}$ is simple.  It is well known that
\begin{align}\label{eqERedges2}
\pr\brk{\G_{n,m}\in\cS}&= \Omega(1)&\mbox{ uniformly for all }m\leq dn/2+\zeta n.
\end{align}
Moreover, providing $\xi$ is chosen small enough, Azuma's inequality and~\eqref{eqLip} imply that
\begin{align}\label{eqERedges3}
\pr\brk{\abs{\log Z_\beta(\G_{n,m})-\Erw\brk{\log Z_\beta(\G_{n,m})}}>\zeta n}&\leq\exp(-6\xi n)&\mbox{ for all }m\leq dn/2+\zeta n.
\end{align}
The estimates \eqref{eqERedges2}--\eqref{eqERedges3} imply that  for all $m\leq dn/2+\zeta n$,
\begin{align}\label{absolute-bound_1}
\pr\brk{\abs{\log Z_\beta(\G_{n,m})-\Erw\brk{\log Z_\beta(\G_{n,m})}}>\zeta n\mid\cS}
	&\leq\exp(-5\xi n).
\end{align}
Since
\begin{equation}
\label{absolute-bound}
|\log Z_\beta(\G_{n,m})|\leq n\log q +m\beta=O(n+m),
\end{equation}
the bound \eqref{absolute-bound_1} shows that for all $m\leq dn/2+\zeta n$,
\begin{align}\label{absolute-bound_2}
\abs{\Erw\brk{\log Z_\beta(\G_{n,m})\mid\cS}-\Erw\brk{\log Z_\beta(\G_{n,m})}}\leq2\zeta n.
\end{align}
Further, because $\GG\mid(\vM=m)$ and $\G_{n,m}\mid\cS$ are identically distributed, \eqref{absolute-bound_1} and \eqref{absolute-bound_2}
show that 
\begin{align}\nonumber
\pr&\bigg[\, \abs{\log Z_\beta(\GG)-\Erw\brk{\log Z_\beta(\GG)\mid\vM=m}}>3\zeta n\, \bigg|\, \vM=m\bigg]%&\leq\exp(-5\xi n)%&\mbox{for all }m\leq dn/2+\zeta n.
=\pr\brk{\abs{\log Z_\beta(\G_{n,m})-\Erw\brk{\log Z_\beta(\G_{n,m})\mid\cS}}>3\zeta n\mid\cS}\\
&\leq\pr\brk{\abs{\log Z_\beta(\G_{n,m})-\Erw\brk{\log Z_\beta(\G_{n,m})}}>\zeta n\mid\cS}
\leq\exp(-5\xi n)
\qquad\qquad\mbox{for all $m\leq dn/2+\zeta n$}.\label{eqERedges4}
\end{align}
Moreover, combining \eqref{eqERedges1}, \eqref{absolute-bound} and \eqref{eqERedges4}, we obtain
\eqref{eqLemma_Azuma}. 

\medskip

To prove the second assertion, we recall that 
$\vm\disteq\Po_{\leq dn/2}((1-\eps)dn/2)$.
We thus obtain the tail bound
\begin{align}\label{eqERedges6}
\pr\brk{\abs{\vm-(1-\eps)dn/2}>\zeta n}&\leq\exp(-2\xi n)
\end{align}
for sufficiently small $\xi$.
Since $\G\mid (\vm=m)$ and $\G_{n,m}$ are identically distributed, \eqref{eqERedges3} yields
\begin{align}\label{eqERedges7}
\pr\bigg[\, \abs{\log Z_\beta(\G)-\Erw[\log Z_\beta(\G)\mid\vm=m]}>\zeta n \bigg|\, \vm=m\, \bigg] &\leq\exp(-3\xi n).
\end{align}
Finally, providing that $\zeta$ is chosen small enough, \eqref{eqLip} and \eqref{eqERedges6} imply that 
\begin{align*}
|\Erw[\log Z_\beta(\G)\mid\vm=m]-\Erw[\log Z_\beta(\G)]|&\leq \delta n/2&\mbox{ for all }(1-\eps)dn/2-\zeta n\leq m\leq (1-\eps)dn/2+\zeta n.
\end{align*}
Therefore, the second part of~\eqref{eqLemma_Azuma} follows from \eqref{eqERedges6} and \eqref{eqERedges7}.

\subsubsection{The random regular graph}\label{Sec_Lemma_Azuma_rr}
We recall that the random regular graph $\GG$ can be constructed via the configuration model by
drawing a perfect matching $\ve_1,\ldots,\ve_{dn/2}$ of the complete graph on the vertex set $V_n\times[d]$ uniformly at random.
To be precise, the sequence $\ve_1,\ldots,\ve_{dn/2}$ is constructed by successively drawing a uniformly random edge $\ve_{i+1}$ that connects two distinct vertices of the complete graph  on $V_n\times [d]$  that are not incident with $\ve_1,\ldots,\ve_i$.
Let $\G'$  be the random multigraph on $[n]$ obtained by inserting for each matching edge $\ve_i=\{(v,h),(w,j)\}$ an edge between $v$ and $w$ and let $\cS$ denote the event that $\G'$ is simple.
It is well known that
\begin{align}\label{eqregedges1}
\pr\brk{\cS}&=\Omega(1);
\end{align}
see, e.g., \cite[\Cor~9.7]{JLR}.
Moreover, $\GG$ is distributed as $\G'$ given $\cS$.

To prove the first inequality we consider the filtration $(\cE_t)_{t\in[dn/2]}$ with $\cE_t$ generated by $\ve_1,\ldots,\ve_t$.
Then the sequence $(\Erw\brk{\log Z_\beta(\G')\mid\cE_t})_{t\in[dn/2]}$ is a Doob martingale.
Moreover, \eqref{eqLip} implies that
\begin{align}\label{eqregedges2}
\abs{\Erw\brk{\log Z_\beta(\G')\mid\cE_t}-\Erw\brk{\log Z_\beta(\G')\mid\cE_{t+1}}}&\leq \beta.
\end{align}
Therefore, Azuma's inequality yields
\begin{align}\label{eqregedges3}
\pr\brk{\abs{\log Z_\beta(\G')-\Erw[\log Z_\beta(\G')]}>\delta n/8}&\leq\exp(-2\xi n).
\end{align}
The first assertion thus follows from \eqref{eqregedges1} and \eqref{eqregedges3}.
Further, we can think of $\G$ as the multigraph obtained by inserting the edges induced by $\ve_1,\ldots,\ve_{\vm}$ only.
Hence, arguing as for (\ref{eqregedges3}) but stopping after $\vm$ steps gives
\begin{align}\label{eqregedges4}
\pr\bigg[\, \abs{\log Z_\beta(\G)-\Erw[\log Z_\beta(\G)\mid\vm]}>\delta n/8\bigg|\, \vm\, \bigg] &\leq\exp(-2\xi n).
\end{align}
Finally, the second assertion follows from \eqref{eqLip}, \eqref{eqERedges6} and \eqref{eqregedges4}.

\subsection{Proof of \Cor~\ref{Cor_Azuma}}\label{Sec_Cor_Azuma}
Given $\delta,\beta>0$ we choose small enough 
$\eps=\eps(\delta,\beta)>0$, $\zeta=\zeta(\delta,\beta,\eps)$, $\xi=\xi(\delta,\beta,\eps,\zeta)$
and assume that $n$ is sufficiently large.
Once more we treat the binomial and the regular models separately.

\subsubsection{The binomial random graph}
We continue to denote the total number of edges of the binomial graph $\GG=\GG(n,d/n)$ by $\vM$ and by
 $\G_{n,\vM}$ the random multigraph obtained by including $\vM$ uniformly and independently chosen edges.
Due to \eqref{eqERedges1} and~\eqref{eqERedges6}, with probability $1-\exp(-\Omega(n))$,
we can obtain $\G_{n,\vM}$ from $\G$ by adding or removing no more than $2\eps dn$ edges.
Hence, provided $\eps$ is small enough, \eqref{eqLip} ensures that
\begin{align}\label{eqCor_Azuma_1}
\abs{\Erw[\log Z_\beta(\G)]-\Erw[\log Z_\beta(\G_{n,\vM})]}\leq 
  2\eps \beta\, dn\, (1 - e^{-\Omega(n)}) + O(n^2)\, e^{-\Omega(n)} \leq \delta n/3.
\end{align}
Furthermore, with $\cS$ the event that $\G_{n,\vM}$ is simple,  $\GG$ is distributed as $\G_{n,\vM}$ given $\cS$.
Therefore, \eqref{eqERedges2} and \eqref{eqERedges3} imply that
\begin{align}\label{eqCor_Azuma_2}
\abs{\Erw[\log Z_\beta(\GG)]-\Erw[\log Z_\beta(\G_{n,\vM})]}&=
\abs{\Erw[\log Z_\beta(\G_{n,\vM})\mid\cS]-\Erw[\log Z_\beta(\G_{n,\vM})]}\leq \delta n/3.
\end{align}
Finally, the assertion follows from \eqref{eqCor_Azuma_1} and \eqref{eqCor_Azuma_2}.

\subsubsection{The random regular graph}
As in \Sec~\ref{Sec_Lemma_Azuma_rr} we denote by $\G'$ the random multigraph with $dn/2$ edges drawn from the configuration model.
By the principle of deferred decisions we can think of $\G'$ as being obtained from $\G$ by adding the missing $dn/2-\vm$ edges.
Hence, provided that $\eps$ is sufficiently small, \eqref{eqLip} implies that
\begin{align}\label{eqCor_Azuma_11}
\abs{\Erw[\log Z_\beta(\G)]-\Erw[\log Z_\beta(\G')]}\leq \delta n/3.
\end{align}
Furthermore, as $\GG$ is distributed as $\G'$ given the event $\cS$, \eqref{eqregedges1} and \eqref{eqregedges3} yield
\begin{align}\label{eqCor_Azuma_12}
\abs{\Erw[\log Z_\beta(\GG)]-\Erw[\log Z_\beta(\G')]}\leq \delta n/3.
\end{align}
The assertion follows from \eqref{eqCor_Azuma_11} and \eqref{eqCor_Azuma_12}.

\section{Interpolation}\label{Sec_Interpolation}

\noindent
In this section we carry out the technical details of the interpolation argument.
\Sec~\ref{Sec_Prop_inter1} contains the proof of \Prop~\ref{Prop_inter1} while
\Sec~\ref{Sec_Prop_inter2} deals with the proof of \Prop~\ref{Prop_inter2}.
Subsequently, \Sec~\ref{Sec_Prop_inter} contains the proof of \Prop~\ref{Prop_inter} and
finally, in \Sec~\ref{Sec_Prop_zero} we prove \Prop~\ref{Prop_zero}.

\subsection{Proof of \Prop~\ref{Prop_inter1}}\label{Sec_Prop_inter1}
A glimpse at~\eqref{eqPoissons} reveals that the random factor graph $\G_0$ consists of constraint nodes $e_1,\ldots,e_{\vm_0}$ and $b_1,\ldots,b_{\vm_0''}$ only.
(See also the left side of Figure~\ref{Fig_01}.)
The constraints $e_1,\ldots,e_{\vm_0}$ are adjacent to the variables $V_n$ but not to $s$, while $b_1,\ldots,b_{\vm_0''}$ are adjacent to $s$ but not to $V_n$.
Consequently, the partition function factorises:
\begin{align*}
Z(\G_0)&=\cY\cdot\cZ,\qquad\mbox{where}\qquad&
	\cY&=\sum_{\sigma_s=1}^\infty\, \gamma(\sigma_s)\, \prod_{i=1}^{\vm_0''}\psi_{b_i}(\sigma_s),&
\cZ&=\sum_{\sigma\in[q]^{V_n}}\, \prod_{i=1}^{\vm_0}\, \psi_{e_i}(\sigma_{\partial e_i}).
\end{align*}
Hence
\begin{align}\label{eqProp_inter1_1}
\Erw\brk{\log Z(\G_0)}&=\Erw\brk{\log \cZ}+\Erw\brk{\log \cY}
\end{align}
and by construction we have
\begin{align}\label{eqProp_inter1_2}
	\Erw\brk{\log \cZ}&=\Erw\brk{\log Z_\beta(\G)}+\aco{O(\vm_0-dn/2)+o(n)}.
\end{align}
Additionally, $Y'$ is distributed as $\cY$ given $\vm_0''=\lfloor dn/2\rfloor$.
Hence, since $\Pr[\vm_0'' > dn/2] = e^{-\Omega(n)}$, we can couple $\cY$ and $Y'$ such that
\begin{align}\label{eqProp_inter1_2a}
\Erw[\log Y'-\log\cY]&=\Erw\brk{\vecone\{\vm_0''\leq dn/2\}\log\sum_{\sigma_s=1}^{\infty}\gamma(\sigma_s)\prod_{1\leq i\leq dn/2-\vm_0''}\psi_{b_i}(\sigma_s)} + e^{-\Omega(n)}.
\end{align}
Since for any $s\in\NN$ we have $\exp(-\beta)\leq\psi_{b_i}(s)\leq1$ for all $i$, 
by \eqref{eqPoissons} and~\eqref{eqProp_inter1_2a} and applying Poisson tail bounds gives
\begin{align}\label{eqProp_inter1_3}
\abs{\Erw[\log Y']-\Erw[\log\cY]}&\leq \eps\beta dn/2+o(n).
\end{align}
Combining \eqref{eqProp_inter1_1}, \eqref{eqProp_inter1_2} and~\eqref{eqProp_inter1_3}, we obtain
\begin{align}\label{eqProp_inter1_4}
\Erw\brk{\log Z(\G_0)}\geq \Erw\brk{\log Z_\beta(\G)}\,  +  \,\Erw[\log Y']-\eps\beta dn/2+o(n).
\end{align}
Finally, the assertion follows from \eqref{eqProp_inter1_4} and \Cor~\ref{Cor_Azuma}.

\subsection{Proof of \Prop~\ref{Prop_inter2}}\label{Sec_Prop_inter2} 
We begin by defining a set $\cC_t$ of variable nodes of $\G_t$, along with a probability distribution $P_t$ on $\cC_t$.
In the binomial case ($\GG$ is the binomial random graph) 
let $\cC_t=V_n$ and let $P_t$ be the uniform distribution on $\cC_t$.
In the regular case ($\GG$ is the random regular graph) let  $\cC_t$ be the set of 
all vertices $v\in V_n$ of degree $d_{\G_t}(v)$ strictly less than $d$ in $\G_t$, and providing that $\cC_t\neq\emptyset$ we define, for all $v\in\cC_t$,
$$P_t(v)=\frac{d-d_{\G_t}(v)}{\sum_{w\in\cC_t}(d-d_{\G_t}(w))}.$$
In both the binomial and the regular case we refer to the elements of $\cC_t$ as {\em cavities}.
Assuming that $\cC_t\neq\emptyset$, we  denote by $\vc_1,\vc_1',\vc_2,\vc_2',\ldots\in\cC_t$ cavities drawn independently from $P_t$.
Note that $\pr(\cC_t=\emptyset)=e^{-\Omega(n)}$.

The proof of \Prop~\ref{Prop_inter2} relies on coupling arguments.
Specifically, we will couple $\G_t$ with three random factor graphs obtained by adding one more constraint of each of the three types of constraints:
\begin{itemize}
\item assuming that $2\vm_t+\vm_t'\leq dn-2$, we obtain $\G_t'$ from 
$\G_t$ by adding one more constraint $e_{\vm_t+1}$ as per \Def~\ref{Def_bin} or~\ref{Def_reg}, respectively; if $2\vm_t+\vm_t'> dn-2$ then we let $\G_t'=\G_t$.
\item assuming that $2\vm_t+\vm_t'<dn$, we obtain $\G_t''$ from 
$\G_t$ by adding one more constraint $a_{\vm_t'+1}$ in accordance with \Def~\ref{Def_bin} or~\ref{Def_reg}, respectively; if $2\vm_t+\vm_t'= dn$ then we let $\G_t''=\G_t$.
\item finally, obtain $\G_t'''$ from $\G_t$ by adding one more constraint $b_{\vm_t''+1}$.
\end{itemize}
The following lemma expresses the derivative of $\Erw\brk{\log Z_\beta(\G_t)}$ in terms of these three enhanced factor graphs.
Let us observe for future reference that
\begin{align}			\label{eqLemma_inter2_1_4a}
\abs{\log Z_\beta(\G_t)}&\leq n\log q+\beta(\vm_t+\vm_t'+\vm_t'')=O(n),
\end{align}
which follows from the fact that $\G_t$ has at most $1+\vm_t+\vm_t'+\vm_t''\leq dn+1$ constraint nodes, and that the weight functions of the constraint nodes $e_i,a_i,b_i$ satisfy $\exp(-\beta)\leq\psi_{e_i},\psi_{a_i},\psi_{b_i}\leq1$.

\begin{lemma}\label{Lemma_inter2_1}
We have
\begin{align}\label{eqLemma_inter2_1}
\frac2{(1-\eps)dn}\frac{\partial}{\partial t}\Erw\brk{\log Z_\beta(\G_t)}=
	-\Erw\brk{\log Z_\beta(\G_t')}+2\, \Erw\brk{\log Z_\beta(\G_t'')}-\Erw\brk{\log Z_\beta(\G_t''')}+o(1).
\end{align}
\end{lemma}
\begin{proof}
Recalling that $\vm_t,\vm_t',\vm_t''$ are distributed as the independent Poisson variables 
$\vM_t,\vM_t',\vM_t''$ from~\eqref{eqPoissons} given the event \aco{$\cM=\{2\vM_t+\vM_t'\leq dn,\,\vM_t+\vM_t'+\vM_t''\leq dn\}$ from \eqref{eqM}}, we see that
\begin{align}\label{eqLemma_inter2_1_1}
\Erw\brk{\log Z_\beta(\G_t)}
	&=\sum_{(m,m',m'')\in\cM}\pr\brk{\vM_t=m,\,\vM_t'=m',\,\vM_t''=m''\mid\cM}
			\Erw\brk{\log Z_\beta(\G_t)\mid\vm_t=m,\,\vm_t'=m'\,\vm_t''=m''}.
\end{align}
The conditional expectation on the right hand side is independent of $t$.
But the means of $\vM_t,\vM_t',\vM_t''$ are governed by $t$.
Hence, we need to differentiate $\pr\brk{\vM_t=m,\,\vM_t'=m',\,\vM_t''=m''\mid\cM}$.
For $(m,m',m'')\in\cM$ we obtain
\begin{align}\nonumber
\frac\partial{\partial t}&\pr\brk{\vM_t=m,\,\vM_t'=m',\,\vM_t''=m''\mid\cM}
	=\frac\partial{\partial t}\frac{\pr\brk{\vM_t=m}\pr\brk{\vM_t'=m'}\pr\brk{\vM_t''=m''}}{\pr\brk\cM}\\
	&=\frac{\pr\brk\cM\frac\partial{\partial t} \left(\pr\brk{\vM_t=m}\pr\brk{\vM_t'=m'}\pr\brk{\vM_t''=m''}
   \right)
				-\pr\brk{\vM_t=m}\pr\brk{\vM_t'=m'}\pr\brk{\vM_t''=m''}\frac\partial{\partial t}\pr\brk\cM}
				{\pr\brk\cM^2}.\label{eqLemma_inter2_1_2}
\end{align}
The product rule yields
\begin{align}
\frac\partial{\partial t}&\pr\brk{\vM_t=m}\pr\brk{\vM_t'=m'}\pr\brk{\vM_t''=m''}
\nonumber\\
	&=\bc{\frac\partial{\partial t}\pr\brk{\vM_t=m}}\,
    \pr\brk{\vM_t'=m'}\pr\brk{\vM_t''=m''}+\pr\brk{\vM_t=m}\bc{\frac\partial{\partial t}\pr\brk{\vM_t'=m'}}\pr\brk{\vM_t''=m''}\nonumber\\
&\qquad {} +\pr\brk{\vM_t=m}\pr\brk{\vM_t'=m'}\bc{\frac\partial{\partial t}\pr\brk{\vM_t''=m''}}\nonumber\\
	&=-\frac{(1-\eps)dn}2\, \bc{\pr\brk{\vM_t=m-1}-\pr\brk{\vM_t=m}}\pr\brk{\vM_t'=m'}\pr\brk{\vM_t''=m''}\nonumber\\
	&\qquad {} +(1-\eps)dn\, \pr\brk{\vM_t=m}\bc{\pr\brk{\vM_t'=m'-1}-\pr\brk{\vM_t'=m'}}\pr\brk{\vM_t''=m''}\nonumber\\
	&\qquad-\frac{(1-\eps)dn}2\, \pr\brk{\vM_t=m}\pr\brk{\vM_t'=m'}\bc{\pr\brk{\vM_t''=m''-1}-\pr\brk{\vM_t''=m''}},
			\label{eqLemma_inter2_1_3a}
\end{align}
\aco{which \csg{simplifies} to
\begin{align}
	\frac\partial{\partial t}&\pr\brk{\vM_t=m}\pr\brk{\vM_t'=m'}\pr\brk{\vM_t''=m''}\nonumber\\
							 &=-\frac{(1-\eps)dn}2\bigg[\pr\brk{\vM_t=m-1}\pr\brk{\vM_t'=m'}\pr\brk{\vM_t''=m''}-2\pr\brk{\vM_t=m}\pr\brk{\vM_t'=m'-1}\pr\brk{\vM_t''=m''}\nonumber\\
	&\qquad\qquad\qquad\qquad {} +\pr\brk{\vM_t=m}\pr\brk{\vM_t'=m'}\pr\brk{\vM_t''=m''-1}\bigg].
			\label{eqLemma_inter2_1_3}
\end{align}}
Moreover, differentiating $-\pr[\cM] = \pr[\cM^c] - 1$ gives
\begin{align}			\label{eqLemma_inter2_1_4}
-\frac\partial{\partial t}\pr\brk\cM&=\sum_{(m,m',m'')\not\in\cM}
			\frac\partial{\partial t}\pr\brk{\vM_t=m}\pr\brk{\vM_t=m'}\pr\brk{\vM_t''=m''}
		=\exp(-\Omega(n)).
\end{align}
\aco{Combining \eqref{eqLemma_inter2_1_1}--\eqref{eqLemma_inter2_1_4} and using $\pr\brk{\cM}=1-\exp(-\Omega(n))$, we obtain
\begin{align}			\nonumber
&-\frac2{(1-\eps)dn}\frac{\partial}{\partial t}\Erw\brk{\log Z_\beta(\G_t)}=o(1)+\\
		&\sum_{\substack{(m,m',m'')\in\cM}}\Erw\brk{\log Z_\beta(\G_t)\mid\vm_t=m,\vm_t'=m',\vm_t''=m''}\bigg(
				\pr\brk{\vM_t=m-1}\pr\brk{\vM_t'=m'}\pr\brk{\vM_t''=m''}\nonumber\\
		&\qquad\qquad\qquad\qquad-2\,
%\Erw\brk{\log Z_\beta(\G_t)\mid\vm_t=m,\vm_t'=m',\vm_t''=m''}
				\pr\brk{\vM_t=m}\pr\brk{\vM_t'=m'-1}\pr\brk{\vM_t''=m''}%\nonumber\\
		+%\Erw\brk{\log Z_\beta(\G_t)\mid\vm_t=m,\vm_t'=m',\vm_t''=m''}
				\pr\brk{\vM_t=m}\pr\brk{\vM_t'=m'}\pr\brk{\vM_t''=m''-1}\bigg). \label{eqLemma_inter2_1_5}
\end{align}
By the principle of deferred decisions, if $(m,m',m'')\in\cM$ then we can think of $\G_t$ given $\vm_t=m,\vm_t'=m',\vm_t''=m''$ as resulting from $\G_t$ given $\vm_t=m-1,\vm_t'=m',\vm_t''=m''$ via the insertion of one more constraint $e_{\vm_t}$.
Therefore, 
\begin{align}\nonumber
\sum_{(m,m',m'')\in\cM}&\Erw\brk{\log Z_\beta(\G_t)\mid\vm_t=m,\vm_t'=m',\vm_t''=m''}
		\pr\brk{\vM_t=m-1}\pr\brk{\vM_t'=m'}\pr\brk{\vM_t''=m''}\\
		&=\Erw\brk{\vecone\{(\vm_t+1,\vm_t',\vm_t'')\in\cM\}\log Z_\beta(\G_t')}.
		\label{eqLemma_inter2_1_6}
\end{align}
The definition~\eqref{eqPoissons} of the Poisson variables ensures that $\pr[(\vm_t+1,\vm_t',\vm_t'')\in\cM]=1-\exp(-\Omega(n))$.
Hence, \eqref{eqLemma_inter2_1_4a} and \eqref{eqLemma_inter2_1_6} yield
\begin{align}\nonumber
\sum_{(m,m',m'')\in\cM}&\Erw\brk{\log Z_\beta(\G_t)\mid\vm_t=m,\vm_t'=m',\vm_t''=m''}
		\pr\brk{\vM_t=m-1}\pr\brk{\vM_t'=m'}\pr\brk{\vM_t''=m''}\\
		&=\Erw\brk{\log Z_\beta(\G_t')}+o(1).
		\label{eqLemma_inter2_1_7}
\end{align}
Similarly,
\begin{align}
\sum_{(m,m',m'')\in\cM}&\Erw\brk{\log Z_\beta(\G_t)\mid\vm_t=m,\vm_t'=m',\vm_t''=m''}
				\pr\brk{\vM_t=m}\pr\brk{\vm_t'=m'-1}\pr\brk{\vM_t''=m''}
		\nonumber\\
		&=\Erw\brk{\log Z_\beta(\G_t'')}+o(1),\label{eqLemma_inter2_1_8}\\
\sum_{(m,m',m'')\in\cM}&\Erw\brk{\log Z_\beta(\G_t)\mid\vm_t=m,\vm_t'=m',\vm_t''=m''}
				\pr\brk{\vM_t=m}\pr\brk{\vM_t'=m'}\pr\brk{\vM_t''=m''-1}\nonumber\\
		&=\Erw\brk{\log Z_\beta(\G_t''')}+o(1).
		\label{eqLemma_inter2_1_9}
\end{align}
Thus, the assertion follows from \eqref{eqLemma_inter2_1_5}, \eqref{eqLemma_inter2_1_7},
	\eqref{eqLemma_inter2_1_8} and~\eqref{eqLemma_inter2_1_9}.
}\end{proof}

\bigskip

\noindent
Let $\fC$ be the event that $|\cC_t|\geq n^{2/3}$.
The choice of the parameters~\eqref{eqPoissons} ensures that
\begin{align}\label{eqLemma_inter2_2_1}
\pr\brk{\fC}=1-\exp(-\Omega(n)).
\end{align}
We proceed to calculate the three expressions on the r.h.s.\ of~\eqref{eqLemma_inter2_1}.
Recall the function $\psi_{\G_t}$ and the Boltzmann distribution 
$\mu_{\G_t}$ which correspond to $\G_t$, 
defined as in (\ref{psiG-def}) and (\ref{eqBoltzmannGeneral}) \aco{with $\vec\gamma$ from \eqref{eqGAMMA}.}
\aco{Also recall the bracket notation from \eqref{eqbck}.}

\begin{lemma}\label{Lemma_inter2_2}
We have
\begin{align*}
	\Erw\brk{\log Z_\beta(\G_t')}-\Erw\brk{\log Z_\beta(\G_t)}
			&=o(1)-\sum_{\ell=1}^\infty\frac{(1-\eul^{-\beta})^\ell}{\ell}
  \, \Erw\brk{\vecone\fC\cdot\scal{\vecone\{\SIGMA_{\vc_1} =\SIGMA_{\vc_2}\}}{\mu_{\G_t}}^\ell}.
\end{align*}
\end{lemma}
\begin{proof}
Since~\eqref{eqLemma_inter2_1_4a} shows that $\log Z_\beta(\G_t),\log Z_\beta(\G_t')=O(n)$, \eqref{eqLemma_inter2_2_1} implies that
\begin{align}\label{eqLemma_inter2_2_2}
\Erw\brk{\log Z_\beta(\G_t')}-\Erw\brk{\log Z_\beta(\G_t)}&=
	\Erw\brk{\vecone\fC\cdot\log Z_\beta(\G_t')}-\Erw\brk{\vecone\fC\cdot\log Z_\beta(\G_t)}+o(1).
\end{align}
Moreover, conditioned on the event $\fC$, the factor graph $\G_t'$ results from $\G_t$ via the addition of a single constraint $e_{\vm_t+1}$.
Denoting by $\vec u,\vec v$ the variable nodes that $e_{\vm_t+1}$ joins, we obtain
\begin{align}
\log Z_\beta(\G_t')-\log Z_\beta(\G_t) =\log\frac{Z_\beta(\G_t')}{Z_\beta(\G_t)}
		&=\log\,\, \sum_{\sigma \in \NN\times [q]^n}\,
\psi_{e_{\vm_t+1}}(\sigma_{\vu},\sigma_{\vv})			\,
\frac{\psi_{\G_t}(\sigma)}
				{Z_\beta(\G_t)} %\nonumber\\
		=\log\scal{\psi_{e_{\vm_t+1}}}{\mu_{\G_t}}.
\label{eqLemma_inter2_2_3}
\end{align}
(Here the sum is over all $\sigma = (\sigma_s,\sigma_{v_1},\ldots, \sigma_{v_n})\in\NN\times [q]^n$,
recalling that $V_n = \{ v_1,\ldots, v_n\}$.)
In particular, \aco{since $\exp(-\beta)\leq\psi_{e_{\vm_t+1}}(\SIGMA)\leq1$,}
we have $-\beta\leq\log Z_\beta(\G_t')-\log Z_\beta(\G_t)\leq 0$.
Further, conditioned on $\fC$,
 the probability that two cavities $\vc_1,\vc_2$ chosen independently with distribution $P_t$ 
coincide is $o(1)$.
Hence, recalling the construction of the probability distribution $P_t$ on the set $\cC_t$ of cavities, we notice that the distribution of the pair $(\vu,\vv)$ and the distribution of the pair $(\vc_1,\vc_2)$ have total variation distance $o(1)$.
Consequently, \eqref{eqLemma_inter2_2_3} yields
\begin{align}\nonumber
\Erw\brk{\vecone\fC\cdot\log Z_\beta(\G_t')}-\Erw\brk{\vecone\fC\cdot\log Z_\beta(\G_t)}
			&=\Erw\brk{\vecone\fC\cdot \log\scal{\psi_{e_{\vm_t+1}}}{\mu_{\G_t}}}\\
&=o(1)+\Erw\brk{\vecone\fC\cdot\log\bc{1-(1-\eul^{-\beta})\scal{\vecone\{\SIGMA_{\vec c_1}=\SIGMA_{\vec c_2 }\}}{\mu_{\G_t}}}}\nonumber\\
&=o(1)-\sum_{\ell=1}^\infty\frac{(1-\eul^{-\beta})^\ell}{\ell}
		\Erw\brk{\vecone\fC\cdot\scal{\vecone\{\SIGMA_{\vec c_1}=\SIGMA_{\vec c_2}\}}{\mu_{\G_t}}^\ell}.
		\label{eqLemma_inter2_2_4}
\end{align}
The assertion follows from \eqref{eqLemma_inter2_2_2} and \eqref{eqLemma_inter2_2_4}.
\end{proof}

\begin{lemma}\label{Lemma_inter2_3}
We have
\begin{align*}
	\Erw\brk{\log Z_\beta(\G_t'')}-\Erw\brk{\log Z_\beta(\G_t)}
	&=o(1)-\sum_{\ell=1}^\infty\frac{(1-\eul^{-\beta})^\ell}{\ell}\, 
	\Erw\brk{\vecone\fC\cdot\scal{\aco{\RHO_{\SIGMA_s,\vm_t'+1}(\SIGMA_{\vc_1})}}{\mu_{\G_t}}^\ell}.
\end{align*}
\end{lemma}
\begin{proof}
Just as in the proof of \Lem~\ref{Lemma_inter2_2} we have
\begin{align}\label{eqLemma_inter2_3_1}
\Erw\brk{\log Z_\beta(\G_t'')}-\Erw\brk{\log Z_\beta(\G_t)}&=
	\Erw\brk{\vecone\fC\cdot\log Z_\beta(\G_t'')}-\Erw\brk{\vecone\fC\cdot\log Z_\beta(\G_t)}+o(1).
\end{align}
Denote by $\vv\in V_n$ the variable node adjacent to the new constraint $\va_{m_t'+1}$ of $\G_t''$.
Then conditioned on $\fC$ we have 
\begin{align*}
\log Z_\beta(\G_t'')-\log Z_\beta(\G_t)&
		=\log\,\, \sum_{\sigma\in\NN\times[q]^{n}}\, 
		\psi_{a_{\vm_t'+1}}(\sigma_{\vv}) \,
             \frac{\psi_{\G_t}(\sigma)}{Z_\beta(\G_t)}
		=\log\scal{\psi_{a_{\vm_t'+1}}}{\mu_{\G_t}}.
\end{align*}
By construction, the variable node $\vv$ is distributed according to $P_t$, the law of $\vc_1$.
Hence,
\begin{align}\nonumber
\Erw\brk{\vecone\fC\cdot\log Z_\beta(\G_t'')}-\Erw\brk{\vecone\fC\cdot\log Z_\beta(\G_t)}&=
\Erw\brk{\vecone\fC\cdot\log\bc{1-(1-\eul^{-\beta})}\aco{\scal{\RHO_{\SIGMA_s,\vm_t'+1}(\SIGMA_{\vec c_1})}{\mu_{\G_t}}}}\\
	&=-\sum_{\ell=1}^\infty\frac{(1-\eul^{-\beta})^\ell}{\ell}\Erw\brk{\vecone\fC\cdot
	\aco{\scal{\RHO_{\SIGMA_s,\vm_t'+1}(\SIGMA_{\vc_1})}{\mu_{\G_t}}^\ell}}.
		\label{eqLemma_inter2_3_3}
\end{align}
Combining \eqref{eqLemma_inter2_3_1} and \eqref{eqLemma_inter2_3_3} completes the proof.
\end{proof}

\begin{lemma}\label{Lemma_inter2_4}
We have
\begin{align*}
	\Erw\brk{\log Z_\beta(\G_t''')}-\Erw\brk{\log Z_\beta(\G_t)}
	&=o(1)-\sum_{\ell=1}^\infty\frac{(1-\eul^{-\beta})^\ell}{\ell}
	\aco{\Erw\brk{\vecone\fC\cdot\scal{\sum_{\tau=1}^q\RHO_{\SIGMA_s,\vm_t''+1}'(\tau)\, \RHO''_{\SIGMA_s,\vm_t''+1}(\tau)}{\mu_{\G_t}}^\ell}.}
\end{align*}
\end{lemma}
\begin{proof}
This follows from similar manipulations as in the proofs of \Lem s~\ref{Lemma_inter2_2} and~\ref{Lemma_inter2_3}.
\end{proof}

\bigskip

\begin{proof}[Proof of \Prop~\ref{Prop_inter2}]
Let
\begin{align*}
\Delta_\ell&=
\Erw\brk{\vecone\fC\cdot\bc{\scal{\vecone\{\SIGMA_{\vc_1}=\SIGMA_{\vc_2}\}}{\mu_{\G_t}}^\ell
		\aco{-2\scal{\RHO_{\SIGMA_s,\vm_t'+1}(\SIGMA_{\vc_1})}{\mu_{\G_t}}^\ell
		+\scal{\sum_{\tau=1}^q\RHO'_{\SIGMA_s,\vm_t''+1}(\tau)\,
			   \RHO_{\SIGMA_s,\vm_t''+1}''(\tau)}{\mu_{\G_t}}^\ell}}.}
\end{align*}
Combining \Lem s~\ref{Lemma_inter2_1}--\ref{Lemma_inter2_4}, we see that
\begin{align}\label{eqProp_inter2_1}
\frac2{(1-\eps)dn}\frac{\partial}{\partial t}\Erw\brk{\log Z_\beta(\G_t)}
	&=o(1)+\sum_{\ell=1}^\infty\frac{(1-\eul^{-\beta})^\ell}\ell\Delta_\ell.
\end{align}
We are going to show that $\Delta_\ell\geq0$ for all $\ell\geq1$;
then the assertion follows from \eqref{eqProp_inter2_1}.

Thus, fix $\ell\geq1$ and let 
$\SIGMA^{(1)},\SIGMA^{(2)},\ldots,\SIGMA^{(\ell)}$ \aco{denote} independent samples from $\mu_{\G_t}$.
Since the expectation of the product of independent random variables equals the product of their expectations, we can rewrite $\Delta_\ell$ as
\begin{align}
\Delta_\ell&=\Erw\brk{\vecone\fC\cdot
\scal{\bc{\prod_{h=1}^\ell\vecone\{\SIGMA_{\vc_1}^{(h)}=\SIGMA_{\vc_2}^{(h)}\}}
	\aco{-2\bc{\prod_{h=1}^\ell\RHO_{\SIGMA_s^{(h)},\vm_t'+1}(\SIGMA_{\vc_1}^{(h)})}}
		+\aco{\bc{\prod_{h=1}^\ell\sum_{\tau=1}^q\RHO'_{\SIGMA_s^{(h)},\vm_t''+1}(\tau) \, \RHO''_{\SIGMA_s^{(h)},\vm_t''+1}(\tau)}}}{\mu_{\G_t}}}\nonumber\\
	&\aco{=\sum_{\tau\in[q]^\ell}
		\Erw\Bigg[\vecone\fC\cdot\Bigg\langle
		\bigg(\prod_{h=1}^\ell\vecone\{\SIGMA_{\vc_1}^{(h)}=\tau_h\}\bigg)\bigg(\prod_{h=1}^\ell\vecone\{\SIGMA_{\vc_2}^{(h)}=\tau_h\}\bigg)
	-2\bigg(\prod_{h=1}^\ell\vecone\{\SIGMA_{\vc_1}^{(h)}=\tau_h\}\bigg)
\bigg(\prod_{h=1}^\ell\RHO_{\SIGMA_s^{(h)},\vm_t'+1}(\tau_h)\bigg)}\nonumber\\
	&\aco{\qquad\qquad\qquad\qquad			+\bigg(\prod_{h=1}^\ell\RHO'_{\SIGMA_s^{(h)},\vm_t''+1}(\tau_h)\bigg)
	\bigg(\prod_{h=1}^\ell\RHO''_{\SIGMA_s^{(h)},\vm_t''+1}(\tau_h)\bigg),\mu_{\G_t}\Bigg\rangle\Bigg].}
\label{eqProp_inter2_2}
\end{align}
To simplify the last expression, we introduce for $\tau\in[q]^\ell$, 
\begin{align*}
\vX_\tau&=\sum_{c\in\cC_t}P_t(c)\prod_{h=1}^\ell\vecone\{\SIGMA_c^{(h)}=\tau_h\}.
\end{align*}
\aco{Further, recall the family $(\hat\RHO_{i})_{i\geq1}$ of distributions $\hat\RHO_{i}\in\cP([q])$ drawn from $\hat\vr\in\cP^2([q])$ from \eqref{eqAllMyRhos}.}
Writing $\Erw'$ for the expectation over $(\hat\RHO_{i})_{i\geq1}$ only, let
\begin{align*}
\vY_\tau&=\Erw'\brk{\prod_{h=1}^\ell\hat\RHO_{\SIGMA_s^{(h)}}(\tau_h)}.
\end{align*}
Since $\vc_1,\vc_2$ and $(\RHO_{s,\vm_t'+1},\RHO_{s,\vm_t''+1}',\RHO_{s,\vm_t''+1}'')_{s\geq1}$ in~\eqref{eqProp_inter2_2} are mutually independent, we can interchange the order in which expectations are taken and rewrite~\eqref{eqProp_inter2_2} as
\begin{align}\label{eqProp_inter2_3}
\Delta_\ell&=\sum_{\tau\in[q]^\ell}\Erw\brk{\vecone\fC\cdot\scal{\vX_\tau^2-2\vX_\tau\vY_\tau+\vY_\tau^2}{\mu_{\G_t}}}
	=\sum_{\tau\in[q]^\ell}\Erw\brk{\vecone\fC\cdot\scal{(\vX_\tau-\vY_\tau)^2}{\mu_{\G_t}}}\geq0.
\end{align}
Finally, the assertion follows from \eqref{eqProp_inter2_1} and \eqref{eqProp_inter2_3}.
\end{proof}

\subsection{Proof of \Cor~\ref{Prop_inter}}\label{Sec_Prop_inter}
We begin by estimating the partition function of $\G_1$.

\begin{lemma}\label{Lemma_0}
For any $\delta>0$ there is $\eps>0$ such that for all large enough $n$ we have
 $\Erw[\log Z(\G_1)]\leq\Erw[\log Y]+\delta n$.
\end{lemma}
\begin{proof}
Let $\vd_1,\vd_2,\ldots,\vd_n$ denote the degrees of the variable nodes 
$v_1,\ldots,v_n$ in $\G_1$.
Each of the constraints $a_j$ is adjacent to only one of the variable nodes
 from $V_n$. For each $v_i$, suppose the constraints $a_{i_1},\ldots, a_{\vd_i}$
are adjacent to $v_i$ and let $\RHO_{\sigma_s,i,h}$ denote the distribution associated
with $a_{i_h}$, for $h\in [\vd_i]$. (In the definition of the interpolation scheme,
this distribution is denoted $\RHO_{\sigma_s,i_h}$.)
Then we can write
\begin{align}\nonumber
\Erw[\log Z(\G_1)]&=\Erw\brk{\log\sum_{\sigma_s=1}^\infty\, \gamma(\sigma_s)\,\,\,
  \sum_{(\sigma_v)_{v\in V_n}\in[q]^{V_n}}\,\,
			\prod_{j=1}^{\vm'}\, \psi_{a_j}(\sigma_{\partial a_j})}\\
				  &=\aco{\Erw\brk{\log\sum_{\sigma_s=1}^\infty\gamma(\sigma_s)\,\, \prod_{i=1}^n\,\, \sum_{\sigma_{v_i}=1}^q\,\, \prod_{j=1}^{\vd_i}\bc{1-(1-\exp(-\beta))\RHO_{\sigma_s,i,j}(\sigma_{v_i})}}.}
		\label{eqLemma_0_1}
\end{align}
Suppose first that $\GG$ is the binomial random graph and let
$\vd_1',\ldots,\vd_n'\disteq\Po((1-\eps)d)$ be independent random variables.
The construction of $\G_1$ ensures that $(\vd_1,\ldots,\vd_n)$ is distributed as 
$(\vd_1',\ldots,\vd_n')$ given
$\vd_1'+\cdots+\vd_n'\leq dn$.
Since this event occurs with probability $1-\exp(-\Omega(n))$, we conclude that
$\dTV((\vd_1,\ldots,\vd_n),(\vd_1',\ldots,\vd_n'))=\exp(-\Omega(n))$.
Therefore, \eqref{eqLemma_0_1} yields
\begin{align}\label{eqLemma_0_2}
	\Erw[\log Z(\G_1)]&=\aco{\Erw\brk{\log\sum_{\sigma_s=1}^\infty\gamma(\sigma_s)\,\, \prod_{i=1}^n\, \sum_{\sigma_{v_i}=1}^q\, \prod_{j=1}^{\vd_i'}\, \bc{1-(1-\eul^{-\beta})\RHO_{\sigma_s,i,j}(\sigma_{v_i})}}+o(n).}
\end{align}
To compare this last expression with $Y$ from \eqref{eqY}, 
let $\vec\Delta_i\sim\Po(\eps d)$ be independent random variables for $i\in [n]$.
Then we can couple the $\vD_i$ from \eqref{eqY} and the $\vd_i'$ from \eqref{eqLemma_0_2} by letting $\vD_i=\vd_i'+\vec\Delta_i$.
Thus, since each factor in \eqref{eqLemma_0_2} lies in the interval $[\exp(-\beta),1]$, we obtain the estimate
\begin{align}\label{eqLemma_0_3}
\Erw[\log Z(\G_1)] 
\leq \Erw[\log Y]+\beta\, \Erw\brk{\sum_{i=1}^n\vec\Delta_i}+o(n)\leq\Erw[\log Y]+\beta\eps dn+o(n),
\end{align}
whence the assertion follows.
Second, if $\GG$ is the random regular graph then $\vD_i=d$ deterministically.
Hence, letting $\vec\Delta_i=d-\vd_i$, we obtain \eqref{eqLemma_0_3} in this case as well.
\end{proof}

\medskip

\begin{proof}[Proof of \Cor~\ref{Prop_inter}]
The corollary {follows} from \Prop~\ref{Prop_inter1}, \Prop~\ref{Prop_inter2} and 
\Lem~\ref{Lemma_0} {by taking $\delta$ to zero}.
\end{proof}

\medskip

\subsection{Proof of \Prop~\ref{Prop_zero}}\label{Sec_Prop_zero}
We will calculate the limits of the two terms appearing in~\eqref{eqProp_PD} separately.
To facilitate a unified treatment, 
let $\fp\in\cP([0,1])$ be the given probability distribution in the binomial case and let $\fp=\delta_\alpha$ for $\alpha\in[0,1]$ in the case of the random regular graph.
Also let $(\ALPHA_i)_{i\geq1}$ be independent samples from $\fp$.

\begin{lemma}\label{Lemma_zero1}
We have
\begin{align*}
\lim_{y\to0}\lim_{\beta\to\infty}
\Erw\brk{\log\, \Erw\brk{\bc{\sum_{\tau=1}^q\prod_{h=1}^{\vD_1}1-(1-\eul^{-\beta}) \RHO_{1,1,h}(\tau)}^y\,\bigg|\,\cR}}
	&=\Erw\brk{\log\sum_{i=0}^{q-1}(-1)^i\bink q{i+1}\prod_{h=1}^{\vD_1}(1-(i+1)(1-\ALPHA_h)/q)}.
\end{align*}
\end{lemma}
\begin{proof}
For $c\in[q]$ let $U_c=\{\forall h\in[\vD_1]: \RHO_{1,1,h}\neq\delta_c\}$ and let $U=\bigcup_{c\in[q]}\, U_c$.
Then
\begin{align*}
0\leq \, \sum_{\tau=1}^q\prod_{h=1}^{\vD_1}1-(1-\eul^{-\beta})\RHO_{1,1,h}(\tau) & \leq \, q\exp(-\beta)\hspace*{12mm} \mbox{if $U$ does not occur},\\
(1-(1-\eul^{-\beta})/q)^{\vD_1}\, \leq \sum_{\tau=1}^q\prod_{h=1}^{\vD_1}1-(1-\eul^{-\beta}) \RHO_{1,1,h}(\tau) &\leq \, q \hspace{25mm} \mbox{if $U$ occurs}.
\end{align*}
(The lower bound in the first line is trivial, while the upper bound follows
since $U_c$ fails for each $c\in [q]$.  The upper bound in the second line is
trivial, while the lower bound follows by taking the term corresponding to some
colour $c$ where $U_c$ holds.)
Consequently, we obtain 
\begin{align}\label{eqLemma_zero1_1}
\lim_{y\to 0}\lim_{\beta\to\infty}\Erw\brk{\bc{\sum_{\tau=1}^q\prod_{h=1}^{\vD_1}1-(1-\eul^{-\beta})
 \RHO_{1,1,h}(\tau)}^y
		\,\bigg|\,\cR}
	&=\pr\brk{U\mid\cR}
\end{align}
pointwise.
Furthermore, by inclusion/exclusion,
\begin{align}\label{eqLemma_zero1_2}
\pr\brk{U\mid\cR}&=\pr\brk{\bigcup_{c\in[q]}\, U_c\,\bigg|\,\cR}=\sum_{k=1}^q(-1)^{k-1}\sum_{Q\subset[q]:|Q|=k}
		\pr\brk{\bigcap_{c\in Q}U_c\,\bigg|\,\cR}.
\end{align}
Since $\RHO_{1,1,1},\ldots,\RHO_{1,1,\vD_1}$ are mutually independent given $\cR$, for any set $Q\subseteq [q]$ of size $k$ we find that
	$$\pr\brk{\bigcap_{c\in Q}U_c\,\bigg|\,\cR}=\prod_{h=1}^{\vD_1}(1-k(1-\ALPHA_h)/q)$$
using (\ref{eqralpha})--(\ref{eqrhoALPHA}).
Hence, \eqref{eqLemma_zero1_2} yields
\begin{align}\label{eqLemma_zero1_3}
\pr\brk{U\mid\cR}&=\sum_{k=1}^q(-1)^{k-1}\bink qk
		\prod_{h=1}^{\vD_1}(1-k(1-\ALPHA_h)/q).
\end{align}
Finally, the assertion follows from \eqref{eqLemma_zero1_1} and \eqref{eqLemma_zero1_3}.
\end{proof}

\begin{lemma}\label{Lemma_zero2}
We have
\begin{align*}
\lim_{y\to0}\lim_{\beta\to\infty}
	\Erw\brk{\log\, \Erw\brk{\bc{1-(1-\eul^{-\beta})\sum_{\tau=1}^q \RHO_{1,1}'(\tau)\RHO_{1,1}''(\tau)}^y\,\bigg|\,\cR}}
	&=\Erw\brk{\log (1-(1-\ALPHA_1)(1-\ALPHA_2)/q)}.
\end{align*}
\end{lemma}
\begin{proof}
Let $U$ be the event that there exists $c\in[q]$ such that $\RHO_{1,1}'=\RHO_{1,1}''=\delta_c$.
Then
\begin{align*}
0\, \leq 1-(1-\eul^{-\beta})\sum_{\tau=1}^q \RHO_{1,1}'(\tau)\RHO_{1,1}''(\tau) &=
    \exp(-\beta) \hspace*{20mm} \mbox{ if $U$ occurs},\\
1-(1-\eul^{-\beta})/q\, \leq 1-(1-\eul^{-\beta})\sum_{\tau=1}^q \RHO_{1,1}'(\tau)\RHO_{1,1}''(\tau) &\leq \, 1 \hspace*{30mm} \mbox{ if $U$ does not occur.}
\end{align*}
(The sum over $\tau$ equals 0 if $\RHO_{1,1}'$ and $\RHO_{1,1}''$ are
atoms on two different colours, and equals $1/q$ otherwise.)
Therefore, we have pointwise convergence
\begin{align}\label{eqLemma_zero2_1}
\lim_{y\to0}\lim_{\beta\to\infty}\Erw\brk{\bc{1-(1-\eul^{-\beta})\sum_{\tau=1}^q
 \RHO_{1,1}'(\tau)\RHO_{1,1}''(\tau)}^y
		\,\bigg|\,\cR}
	&=1-\pr\brk{U\mid\cR}.
\end{align}
Since $\pr\brk{U\mid\cR}=(1-\ALPHA_1)(1-\ALPHA_2)/q$,
using (\ref{eqralpha})--(\ref{eqrhoALPHA}),
the assertion follows from \eqref{eqLemma_zero2_1}.
\end{proof}

\begin{proof}[Proof of \Prop~\ref{Prop_zero}]
The proposition follows from \Lem s~\ref{Lemma_zero1} and~\ref{Lemma_zero2} immediately.
\end{proof}

\section{Asymptotics}\label{Sec_asymptotics}

\noindent
We perform asymptotic expansions of $\Sigma_{d,q}(\nix)$, $\Sigma_{d,q}^*(\nix)$ in the limit of large $q$ to prove Corollaries~\ref{Cor_reg} and~\ref{Cor_ER}.
In this section, the notation $\tilde O_q(\cdot)$ suppresses polynomials in 
$\log q$,
and both $O_q(\cdot)$ and $\tilde O_q(\cdot)$ refer to the limit $q\to\infty$.

\subsection{Proof of \Cor~\ref{Cor_reg}}

\medskip
\noindent
Write $\Sigma_{d,q}(\alpha) = S - T$, where
\[ S=\log\sum_{i=0}^{q-1}(-1)^i\bink q{i+1}(1-(i+1)(1-\alpha)/q)^d,
\qquad T = \frac{d}{2}\log\bc{1-(1-\alpha)^2/q}.
\]
We will let
\begin{align}\label{eqCor_reg_0}
\alpha&=\frac1{2q},&d=(2q-1)\log q-c
\end{align}
with  $c=O_q(1)$, and expand $S$ and $T$ %we expand $\Sigma_{d,q}(\alpha)$ 
asymptotically in the limit $q\to\infty$.
Substituting for $d$ in $S$ gives
\begin{equation}
S= \log\sum_{i=0}^{q-1}(-1)^i\bink q{i+1}\exp\bc{((2q-1)\log q-c)\log\bc{1-(i+1)(1-\alpha)/q}}.
	\label{eqCor_reg_1}
\end{equation}
Observe that
\begin{equation}
\log(1-(1-\alpha)/q) = - \frac{1}{q} + O(q^{-3}),
\label{alpha-useful}
\end{equation}
and so, for the $i=0$ term we have the expansion
\begin{align}\nonumber
\exp\bc{((2q-1)\log q-c)\log\bc{1-(1-\alpha)/q}}&=\exp\bc{-((2q-1)\log q-c)/q+\tilde O_q(q^{-2})}\\
	&=\exp\bc{-2\log q+(\log q)/q+c/q+\tilde O_q(q^{-2})}.
	%&=q^{-2}\bc{1+(\log q)/q+c/q+\tilde O_q(q^{-2})}.
\label{eqCor_reg_2}
\end{align}
Moreover, for $i\geq 1$ we have
\begin{align}\label{eqCor_reg_3}
\exp\bc{((2q-1)\log q-c)\log\bc{1-(i+1)(1-\alpha)/q}}&
	= q^{-2(i+1)}\bc{1 + \tilde O_q\bc{\frac{1}{q} + \frac{(i+1)^2}{q^2}}}.
\end{align}
Plugging \eqref{eqCor_reg_2} and \eqref{eqCor_reg_3} into \eqref{eqCor_reg_1} gives
\begin{align*}
S&= \log\bc{ q\cdot \exp\bc{-2\log q+(\log q)/q+c/q+\tilde O_q(q^{-2})}
   - \nfrac{1}{2} q^{-2}  + \tilde O_q(q^{-2})) }\\
&= -\log q+\frac{\log q}{q}+\frac{c}{q} + \tilde O_q(q^{-2})
 + \log\bc{1 - 1/(2q) + \tilde O_q(q^{-2})}\\
&= -\log q+\frac{\log q}{q}+\frac{2c-1}{2q} + \tilde O_q(q^{-2}).
\end{align*}
Similarly, substituting for $\alpha$ and $d$ in $T$ gives
\[
T =\bc{q\log q-\nfrac{1}{2}\log q-c/2}\cdot\bc{-1/q+1/(2q^2)+O_q(q^{-3})} 
	=-\log q+ \frac{\log q}{q}+\frac c{2q}+\tilde O_q(q^{-2}).
	\label{eqCor_reg_4}
\]
Hence,
\begin{align*}
\Sigma_{d,q}(\alpha)=S-T&=\frac{c-1}{2q}+\tilde O_q(q^{-2}).
\end{align*}
Consequently, if $c\leq1-\eps_q$ where $\eps_q\to0$ slowly enough then $\Sigma_{d,q}(\alpha)<0$ for large enough $q$.
This completes the proof of \Cor~\ref{Cor_reg}.

\subsection{Proof of \Cor~\ref{Cor_ER}}

\noindent
With $\alpha,d$ as in \eqref{eqCor_reg_0} we consider the distribution $\fp=\delta_\alpha$.
With $\vD\disteq\Po(d)$ let
\begin{align}\label{eqCor_ER_1}
S&=\Erw\brk{\log\sum_{i=0}^{q-1}(-1)^i\bink q{i+1}(1-(i+1)(1-\alpha)/q)^{\vD}},&T=\frac d2\log\bc{1-(1-\alpha)^2/q}.
\end{align}
First,
\begin{equation}
\label{rest}
 (2q)^{-\vec D}\leq\sum_{i=0}^{q-1}(-1)^i\bink q{i+1}(1-(i+1) (1-\alpha)/q)^{\vD}\leq 1.\end{equation}
To see this, we interpret the sum in the middle as an inclusion/exclusion formula.
Namely, choose $\vec c_1,\ldots,\vec c_{\vD}\in\{0,1,\ldots,q\}$ independently such that the probability of drawing $0$ equals $\alpha$ and the probability of drawing $i\in[q]$ equals $(1-\alpha)/q$.
Then the sum equals the probability of the event $[q]\setminus\cbc{\vec c_1,\ldots,\vec c_{\vD}}\neq\emptyset$, which is clearly lower bounded by $\alpha^{\vD}=(2q)^{-\vD}.$
{Poisson tail bounds show that 
  $\Pr[\, |\vD - d| \geq 10\sqrt{q}\log q] = O_q(q^{-4})$
and combining this with (\ref{rest}) gives}
%Hence,  Bennett's inequality applied to the Poisson random variable $\vD$ 
%(viewed as the sum of $\lceil d\rceil$ independent $\Po(d/\lceil d\rceil)$ variables) implies that
\begin{align*}
S&=\Erw\brk{\log\sum_{i=0}^{q-1}(-1)^i\bink q{i+1}(1-(i+1)(1-\alpha)/q)^{\vD}\,\bigg|\,\abs{\vD-d}\leq 10\sqrt q\log q}
		+O_q(q^{-2}).
\end{align*}
%\csg{[[CSG:I think standard Poisson tail bounds suffice to show 
%$\Pr(|\vD - d|\leq 10\sqrt{q}\log q) = 1-O(q^{-4})$, right?  In fact I think we
%could show $O_q(q^{-24})$  or so (?). The problem before
%was that we didn't have the factor of 10.]]} 
Hence, let $\vec\Delta$ be distributed as $\vD-d$ given $\abs{\vD-d}\leq10\sqrt q\log q$.
Then 
\begin{align}\nonumber
S&=\Erw\brk{\log\sum_{i=0}^{q-1}(-1)^i\bink q{i+1}(1-(i+1)(1-\alpha)/q)^{d+\vec\Delta}}+O_q(q^{-2})\\
	&=\Erw\brk{\log\sum_{i=0}^{q-1}(-1)^i\bink q{i+1}\exp((d+\vec\Delta)\log(1-(i+1)(1-\alpha)/q))}+O_q(q^{-2}).
	\label{eqCor_ER_2}
\end{align}
For the $i=0$ term, using (\ref{eqCor_reg_0}) and (\ref{alpha-useful}), 
we have the expansion
\begin{align}\nonumber
\exp((d+\vec\Delta)\log(1-(1-\alpha)/q))&=
	\exp\bc{-(d+\vec\Delta)/q+\tilde O_q(q^{-2})}\\
	&=\exp\bc{-2\log q+\log q/q+c/q-\vec\Delta/q+\tilde O_q(q^{-2})}\\
	&=q^{-2}\bc{1+\log q/q+c/q+\tilde O_q(q^{-2})}\exp(-\vec\Delta/q).\label{eqCor_ER_3}
\end{align}
Moreover, for $i\geq1$,
using the fact that $\vec\Delta/q=\tilde O(q^{-1/2})$, we obtain
\begin{align}\label{eqCor_ER_4}
\exp\bc{(d+\vec\Delta)\log\bc{1-(i+1)(1-\alpha)/q}}&
	= q^{-2(i+1)}\bc{1 + \tilde O_q\bc{\frac{1}{q^{1/2}} + \frac{(i+1)^2}{q^2}}}.
\end{align}
Plugging \eqref{eqCor_ER_3} and \eqref{eqCor_ER_4} into \eqref{eqCor_ER_2}, we obtain
\begin{align}
S&=\Erw\brk{\log\bc{q^{-1}\bc{1+\log q/q+c/q+\tilde O_q(q^{-2})}\exp(-\vec\Delta/q)
		-\nfrac{1}{2}q^{-2} +\tilde O_q(q^{-5/2})}} + O_q(q^{-2})\nonumber\\
&=-\log q+ \log\bc{1+\log q/q+c/q+\tilde O_q(q^{-2})} - \Erw\brk{\vec\Delta/q}
		+ \log\bc{1-1/(2q)+\tilde O_q(q^{-3/2})}\nonumber\\
&=-\log q+\frac{\log q}{q}+\frac{2c-1}{q} - \Erw\brk{\vec\Delta/q} 
  +\tilde O_q(q^{-3/2}).
	\label{eqCor_ER_5}
\end{align}
%Now
%\begin{align}\label{eqCor_ER_6}
%\exp(-\vec\Delta/q)&=1-\frac{\vec\Delta}q+\frac12\bcfr{\vec\Delta}q^2 \csg{+ \tilde O_q(q^{-3/2})},\qquad\mbox{and thus}&
%\bc{\exp(-\vec\Delta/q)-1}-\frac12\bc{\exp(-\vec\Delta/q)-1}^2&=-\frac{\vec\Delta}q+\tilde O(q^{-3/2}).
%\end{align}
Since $d=\Erw[\vD]$ and since conditioning on $\abs{\vD-d}\leq10\sqrt q\log q$ does not shift 
the mean of $\vD$ by more than $O_q(1/q)$, we obtain $\Erw[\vec\Delta]=O_q(1/q)$.
Using this and \eqref{eqCor_ER_5} yields
\begin{align}\label{eqCor_ER_7}
S&=-\log q+\frac{\log q}{q}+\frac{2c-1}{2q}+ \tilde O_q(q^{-3/2}).
\end{align}
Combining \eqref{eqCor_ER_7} with the expansion \eqref{eqCor_reg_4} of $T$, we finally
obtain
\begin{align*}
\Sigma^*_{d,q}(\delta_\alpha)&=S-T=\frac{c-1}{2q}+\tilde O_q(q^{-3/2}).
\end{align*}
Thus, setting $c\leq1-\eps_q$ with $\eps_q\to0$ slowly, we see that $\Sigma^*_{d,q}<0$ for large enough $q$. This completes the proof of Corollary~\ref{Cor_ER}.

\subsubsection*{Acknowledgment.}
\aco{We thank Viktor Harangi and an anonymous reviewer for their very careful reading of our manuscript and their extremely accurate and helpful comments, which led to several improvements and corrections.}


\begin{thebibliography}{29}

\bibitem{AchFried}
D.~Achlioptas, E.~Friedgut:
A sharp threshold for $k$-colorability.
Random Struct.\ Algorithms {\bf 14} (1999) 63--70.

\bibitem{AMo}
D.~Achlioptas, C.~Moore:
Almost all graphs with average degree 4 are 3-colorable.
Journal of Computer and System Sciences {\bf 67} (2003) 441--471.

\bibitem{AchMoore}
D.\ Achlioptas, C.\ Moore: The chromatic number of random regular graphs. Proc.\ 8th RANDOM (2004) 219--228.

\bibitem{AchNaor}
D.~Achlioptas, A.~Naor:
The two possible values of the chromatic number of a random graph.
Annals of Mathematics {\bf 162} (2005) 1333--1349.

\bibitem{AlonKriv}
N.~Alon, M.~Krivelevich: The concentration of the chromatic number of random graphs.
Combinatorica {\bf 17} (1997) 303--313

\bibitem{Cond}
V.\ Bapst, A.\ Coja-Oghlan, S.\ Hetterich, F.\ Rassmann, D.\ Vilenchik:
The condensation phase transition in random graph coloring.
Communications in Mathematical Physics {\bf341} (2016) 543--606.

\bibitem{bayati}
M.~Bayati, D.~Gamarnik, P.~Tetali:
Combinatorial approach to the interpolation method and scaling limits in sparse random graphs.
Annals of Probability {\bf41} (2013) 4080--4115.

\bibitem{BBColor}
B.~Bollob\'as: The chromatic number of random graphs.
Combinatorica {\bf8} (1988) 49--55

\bibitem{ACOcovers}
A.~Coja-Oghlan: Upper-bounding the $k$-colorability threshold by counting covers. 
Electronic Journal of Combinatorics {\bf 20} (2013) P32 

\bibitem{ColReg}
A.~Coja-Oghlan, C.~Efthymiou, S.~Hetterich: On the chromatic number of random regular graphs.
Journal of Combinatorial Theory, Series B {\bf116} (2016) 367--439.

\bibitem{CG}
A.\ Coja-Oghlan, A.\ Erg\"ur, P.\ Gao, S.\ Hetterich, M.\ Rolvien: The rank of sparse random matrices. Proc.~31st SODA (2020) 579--591.

\bibitem{CKPZ}
A. Coja-Oghlan, F. Krzakala, W. Perkins and  L. Zdeborova:
Information-theoretic thresholds from the cavity method. 
Advances in Mathematics {\bf 333} (2018) 694--795.

\bibitem{Kosta}
A.~Coja-Oghlan, K.~Panagiotou: The asymptotic $k$-SAT threshold.
Advances in Mathematics {\bf288} (2016) 985--1068.

\bibitem{Angelika}
A.~Coja-Oghlan, K.~Panagiotou, A.~Steger:
On the chromatic number of random graphs.
Journal of Combinatorial Theory, Series B, {\bf 98} (2008) 980--993.

\bibitem{COP}
A.~Coja-Oghlan, W.~Perkins:
Spin systems on Bethe lattices.
Communications in Mathematical Physics {\bf372} (2019) 441--523.

\bibitem{ACOVilenchik}
A.~Coja-Oghlan, D.~Vilenchik:
The chromatic number of random graphs for most average degrees. International Mathematics Research Notices {\bf 2016} (2016) 5801--5859.

\bibitem{CFRR}
C.\ Cooper, A.\ Frieze, B.\ Reed, O.\ Riordan:
Random regular graphs of non-constant degree: independence and chromatic number.
Comb.\ Probab.\ Comput.\ {\bf 11} (2002) 323--341.

\bibitem{DKKKPW}
J.~Diaz, A.~Kaporis, G.~Kemkes, L.~Kirousis, X.~P\'erez, N.~Wormald:
On the chromatic number of a random 5-regular graph.
Journal of Graph Theory {\bf 61} (2009) 157--191.
 
\bibitem{DSS1}
J.~Ding, A.~Sly, N.~Sun: Satisfiability threshold for random regular NAE-SAT.
Communications in Mathematical Physics {\bf 341} (2016) 435--489.

\bibitem{DSS2}
J.~Ding, A.~Sly, N.~Sun:
Maximum independent sets on random regular graphs.
Acta Math.\ {\bf217} (2016) 263--340.

\bibitem{DSS3}
J.~Ding, A.~Sly, N.~Sun: Proof of the satisfiability conjecture for large $k$.
Proc.\ 47th STOC (2015) 59--68.

\bibitem{DuboisMandler}
O.~Dubois, J.~Mandler:
On the non-3-colourability of random graphs
arXiv:math/0209087 (2002).

\bibitem{ER}
P.\ Erd\H os, A.\ R\'enyi: On the evolution of random graphs.
Magayar Tud.\ Akad.\ Mat.\ Kutato Int.\ Kozl.\ {\bf 5} (1960) 17--61.

\bibitem{FranzLeone}
S.\ Franz, M.\ Leone: Replica bounds for optimization problems and diluted spin systems.
 J.\ Stat.\ Phys.\ {\bf111} (2003) 535--564.

\bibitem{FriezeLuczak}
A.~Frieze, T.~\Luczak: On the independence and chromatic numbers of random regular graphs.
J.\ Comb.\ Theory B {\bf 54} (1992) 123--132.

\bibitem{JLR} 
S.~Janson, T.~{\L}uczak, A.~Ruci\'nski: Random Graphs. Wiley  (2000)

\bibitem{Guerra}
F.~Guerra: Broken replica symmetry bounds in the mean field spin glass model. Comm.\ Math.\ Phys.\ {\bf 233} (2003) 1--12.

\bibitem{KPGW}
G.~Kemkes, X.~P\'erez-Gim\'enez, N.~Wormald:
On the chromatic number of random $d$-regular graphs.
Advances in Mathematics {\bf 223}  (2010) 300--328.

\bibitem{KSVW}
M.\ Krivelevich, B.\ Sudakov, V.\ Vu, N.\ Wormald: Random regular graphs of high degree. Random Struct.\  Algor.\ {\bf18} (2001) 346--363.

\bibitem{Lelarge}
M.\ Lelarge, M.\ Oulamara:
Replica bounds by combinatorial interpolation for diluted spin systems.
J.\ Stat.\ Phys {\bf 173} (2018) 917--940.

\bibitem{LuczakColor} % builds upon Matula's idea
T.~{\L}uczak: The chromatic number of random graphs.
Combinatorica {\bf11} (1991) 45--54

\bibitem{Matula}
D.\ Matula: Expose-and-merge exploration and the chromatic number of a random graph.
Combinatorica {\bf 7} (1987) 275--284.

\bibitem{MatulaKucera}
D.\ Matula, L. Ku\v cera: An expose-and-merge algorithm and the chromatic number of a random graph.
Proc.\ Random Graphs 87 (1987) 175--187.

\bibitem{MM}
M.~M\'ezard, A.~Montanari:
Information, Physics and Computation.
Oxford University Press~2009.

\bibitem{MP}
\aco{M.~\Mezard, G.\ Parisi: The Bethe lattice spin glass revisited.
European Physical Journal B {\bf20} (2001) 217--233.}

\bibitem{Panchenko}
D.\ Panchenko: The Sherrington-Kirkpatrick model. Springer 2013.

\bibitem{Panchenko2}
D.\ Panchenko:
Spin glass models from the point of view of spin distributions.
Annals of Probability {\bf 41}  (2013) 1315--1361.

\bibitem{PanchenkoTalagrand}
D.\ Panchenko, M.\ Talagrand:
Bounds for diluted mean-fields spin glass models.
Probab.\ Theory Relat.\ Fields {\bf130} (2004) 319--336.

\bibitem{ShamirSpencer}
E.~Shamir, J.~Spencer: Sharp concentration of the chromatic number of random graphs $G_{n,p}$.
Combinatorica {\bf 7} (1987) 121--129

\bibitem{Shi1}
L.~Shi, N.~Wormald: Colouring random 4-regular graphs.
Combinatorics, Probability and Computing {\bf16} (2007) 309--344.

\bibitem{Shi2}
L.~Shi, N.~Wormald: Colouring random regular graphs.
Combinatorics, Probability and Computing {\bf16} (2007) 459--494.

\bibitem{SSZ}
A.~Sly, N.~Sun, Y.~Zhang: The number of solutions for random regular NAE-SAT.
Proc.\ 57th FOCS (2016) 724--731; full version available as arXiv:1604.08546.

\bibitem{Talagrand}
\aco{M.~Talagrand: Spin glasses: a challenge for mathematicians. Springer (2003).}

\bibitem{ZK}
L.\ Zdeborov\'a, F.\ Krzakala: Phase transitions in the coloring of random graphs. Phys.\ Rev.\ E {\bf 76} (2007) 031131.

\end{thebibliography}
\end{document}